\documentclass[11pt]{article}

\RequirePackage[a4paper, left=25mm, right=25mm, top=25mm, bottom=30mm]{geometry}
\RequirePackage[OT1]{fontenc}
\RequirePackage[british]{babel}
\RequirePackage{amssymb, amsmath, amsthm}
\allowdisplaybreaks[2]
\usepackage{rotating}
\RequirePackage[sort&compress, numbers]{natbib} %Customize references, add 'numbers' in argument for numbered references. [sort&compress] sorts multiple references to the order of the bibliography.

\DeclareRobustCommand{\NLprefix}[3]{#2}
\RequirePackage{color}
\RequirePackage{hyperref}
\RequirePackage{enumerate}
\RequirePackage[margin=1cm]{caption}

\numberwithin{equation}{section}

\theoremstyle{plain}
\newtheorem{theorem}{Theorem}[section]
\newtheorem{definition}[theorem]{Definition}%[section]
\newtheorem{lemma}[theorem]{Lemma}%[section]
\newtheorem{proposition}[theorem]{Proposition}%[section]
\newtheorem{corollary}[theorem]{Corollary}%[section]
\newcounter{remarknum}
\makeatletter
\def\remark{\@ifnextchar[{\@with}{\@without}}
\def\@with[#1]{\refstepcounter{remarknum}\textit{Remark~\theremarknum.} \label{rem:#1}}
\def\@without{\refstepcounter{remarknum}\textit{Remark~\theremarknum.} }
\makeatother

\newcommand{\numberthis}{\stepcounter{equation}\tag{\theequation}}
\newcommand{\alert}[1]{\textcolor{red}{#1}} %red text: remark/question
\renewcommand{\cite}{\alert{cite[?]}}
\renewcommand\citet\Citet
\renewcommand\citep\Citep
\renewcommand\citeauthor\Citeauthor

\newcommand{\levy}{L{\'e}vy }

\renewcommand{\a}{\alpha}

\renewcommand{\b}{\beta}
\renewcommand{\d}{\delta}
\newcommand{\dd}{\,\mathrm{d}}

\newcommand{\e}{\varepsilon}
\newcommand{\E}{\mathbb{E}}
\newcommand{\g}{\gamma}
\renewcommand{\l}{\lambda}

\newcommand{\N}{\mathbb{N}}
\renewcommand{\O}{O}
\renewcommand{\P}{\mathbb{P}}
\renewcommand{\r}{\rho}

\newcommand{\R}{\mathbb{R}}

\newcommand{\Var}{\mathbb{V}\mathrm{ar}}

\DeclareMathOperator{\erfc}{Erfc}
\DeclareMathOperator{\Exp}{Exp}

\newcommand{\barF}{\overline{F}}
\newcommand{\barG}{\bar{G}}
\newcommand{\inv}[1]{{#1}^{\leftarrow}}
\renewcommand{\Finv}{\inv{F}}
\newcommand{\Ginv}{\inv{G}}
\newcommand{\Uinv}{\inv{U}}

\newcommand{\rhotoone}{{\r\uparrow 1}}

\newcommand{\xtoinfty}{{x\rightarrow\infty}}

\renewcommand{\bar}{\overline}
\renewcommand{\hat}{\widehat}
\renewcommand{\tilde}{\widetilde}

\def\mg{\mathrm{M/GI/1}}
\def\mm{\mathrm{M/M/1}}
\def\fb{\mathrm{FB}}
\def\fifo{\mathrm{FIFO}}
\def\srpt{\mathrm{SRPT}}

\newcommand{\MDA}{\mathrm{MDA}}
\newcommand{\RV}{\mathrm{RV}}

\begin{document}
\title{Heavy-Traffic Analysis of Sojourn Time under the Foreground-Background Scheduling Policy}
\author{
\begin{tabular}{cccccc}
Bart Kamphorst\textsuperscript{a} &&& Bert Zwart\textsuperscript{a,b} \\
\texttt{b.kamphorst@cwi.nl} &&& \texttt{b.zwart@cwi.nl}
\end{tabular} \\
\begin{tabular}{l}
\footnotesize{\textsuperscript{a} Centrum Wiskunde \& Informatica, P.O. Box 94079, 1090 GB Amsterdam, the Netherlands} \\
\footnotesize{\textsuperscript{b} Technische Universiteit Eindhoven, P.O. Box 513, 5600 MB Eindhoven, the Netherlands}
\end{tabular}
}
\date{\small{\today}}
\maketitle

\begin{abstract}
We consider the steady-state distribution of the sojourn time of a job entering an $\mg$ queue with the foreground-background scheduling policy in heavy traffic. The growth rate of its mean, as well as the limiting distribution, are derived under broad conditions. Assumptions commonly used in extreme value theory play a key role in both the analysis and the results.
\end{abstract}

%Proefschrift: Use \varphi for normal pdf; use \phi for functions that vanish.

\section{Introduction}
One of the main insights from queueing theory is that the queue length and sojourn time are of the order $1/(1-\rho)$, as the traffic intensity of the system $\rho$ approaches 100 percent utilization. This insight dates back to \citet{kingman1961single} and \citet{prokhorov1963transition} and, appropriately reformulated, remains valid for queueing networks and multiple server queues \citep{whitt2002stochastic, garmanik2006validity, braverman2017heavy}. This picture can change dramatically when the scheduling policy is no longer First-In-First-Out ($\fifo$). \citet{bansal2005average} was the first to point out that the mean sojourn time (a.k.a.\ response time, flow time) of a user is of $o(1/(1-\rho))$ in the $\mm$ queue when the scheduling policy is Shortest Remaining Processing Time ($\srpt$). This result was later generalized to several non-exponential service time distributions in \citet{lin2011heavy}. More recently, \citet{puha2015diffusion} derived a process limit theorem for the $\srpt$ queue length, under an assumption that implies that all moments of the service time are finite.

As $\srpt$ requires information on service times in advance, the question was raised if the same growth rate in heavy traffic can be reached with a blind scheduling policy, a question that was answered negatively in \citet{bansal2018achievable}. Specifically, the authors showed that for every blind scheduling policy, there exists a service time distribution under which the growth rate in heavy traffic of the mean sojourn time is at least a factor $\log(1/(1-\rho))$ larger than the growth rate of $\srpt$. \citeauthor{bansal2018achievable} also construct a scheduling policy that achieves this growth rate, but this policy rather complicated as it involves randomization. All of the results mentioned thus far only concern the mean sojourn time, and it is of interest to obtain information about the distribution of the sojourn time as well.

Motivated by these developments, we consider the Foreground-Background ($\fb$) scheduling policy in this paper. More precisely, we investigate the invariant distribution of the sojourn time of a customer in an $\mg/\fb$ queue. The $\fb$ policy operates as follows: priority is given to the customer with the least-attained service, and when multiple customers satisfy this property, they are served at an equal rate. The only heavy-traffic results for $\fb$ we are aware of are of ``big-$\O$'' type and are known in case of deterministic, exponential, Pareto and specific finite-support service times \citep{bansal2006handling, nuyens2008foreground}. For deterministic service times, it is easy to see that all customers under $\fb$ depart in one batch at the end of every busy period, and as a result the growth rate in heavy traffic is very poor in this case $\O((1-\rho)^{-2})$. The behaviour of $\fb$ is much better for service-time distributions with a decreasing failure rate, as $\fb$ then optimizes the mean sojourn time among all blind policies \citep{righter1989scheduling}. For more background on the $\fb$ policy we refer to the survey by \citet{nuyens2008foreground}.

The main results of this paper are of three types:
\begin{enumerate}
\item We characterize the exact growth rate (up to a constant independent of $\rho$) of the sojourn time in heavy traffic under very general assumptions on the service time distribution. As in \citet{bansal2006handling} and \citet{lin2011heavy}, we find a dichotomy: when the service time distribution has finite variance, the mean sojourn time $\E[T_\fb^\r] = \Theta\left(\frac{\barF(\Ginv(\r))}{(1-\r)^2}\right)$. Here $\barF(x)=1-F(x)$ is the tail of the service time distribution and $\Ginv$ is the right-inverse of the distribution function of a residual service time; a detailed overview of notation can be found in Section~\ref{subsec:matuszewska}. In the infinite variance case, we find that $\E[T_\fb^\r] = \Theta\left(\log \frac{1}{1-\r}\right)$. This result is formally stated in Theorem~\ref{thm:ETMatuszewska}. The precise conditions for these results to hold involve Matuszewska indices, a concept that will be reviewed in Section~\ref{sec:preliminaries}. The behaviour of $\barF(\Ginv(\r))$ is quite rich, as will be illustrated by several examples.

\item Contrary to the results in \citet{bansal2006handling} and \citet{lin2011heavy}, we have been able to obtain a more precise estimate of the growth rate of $\E[T_\fb^\r]$. It turns out that extreme value theory plays an essential role
in our analysis, and the limiting constant factor in front of the growth rate $\frac{\barF(\Ginv(\r))}{(1-\r)^2}$ crucially depends on in which domain of attraction the service-time distribution is. This result is summarized in Theorem~\ref{thm:ETextreme} and appended in Theorem~\ref{thm:TmeanGumbel}. When the service-time distribution tail is regularly varying, it is shown that the growth rate of the sojourn time under $\fb$ is equal to that of $\srpt$ up to a finite constant. A comparison of the sojourn times under $\fb$ and $\srpt$ is given in Corollary~\ref{cor:FBvsSRPT}.

\item When analysing the distribution, we first show that $T_\fb^\r/\E[T_\fb^\r]$ converges to zero in probability as $\rhotoone$. To still get a heavy traffic approximation for $\P(T_\fb^\r>y)$, we state a sample path representation for the sojourn time distribution for a job that requires a known amount of service. We then use fluctuation theory for spectrally negative \levy processes to rewrite this representation into an expression that is amenable to analysis; in particular, we obtain a representation for the Laplace transform of the \textit{residual} sojourn time distribution is of independent interest, from which a heavy-traffic limit theorem follows. Finally, the Laplace transform implies an estimate for the tail of $T_\fb$. 

More specifically, our results show that $\P((1-\r)^2 T_\fb > y)/ \barF(\Ginv(\r))$ converges to a non-trivial function $g^*(y)$, for which we give an integral expression in terms of error functions. Along the way, we derive a heavy traffic limit for the total workload in an $\mg$ queue, with truncated service times that also seems to be of independent interest (see Proposition~\ref{prop:WtoExp}). As in the analysis for the mean sojourn time, ideas from extreme value theory play an important role in the analysis, and the limit function $g^*$ depends on which domain of attraction the service-time distribution falls into. A precise description of this result can be found in Theorem~\ref{thm:Ttail}.
\end{enumerate}

Despite the fact that extreme theory appears both in our analysis and our end results, the precise role of extreme value theory is not entirely clear from the analysis in this paper. A challenging topic for future research is to get a completely probabilistic proof of our result; e.g.\ a proof that does not use explicit integral expressions for the mean sojourn time. To this end, it makes sense to consider a more general class of scheduling policies, for example the class of SMART scheduling policies considered in \citet{wierman2005nearly} and \citet{nuyens2008preventing}. This is beyond the scope of the present paper.

The rest of the paper is organized as follows. Section~\ref{sec:preliminaries} formally introduces the model that is considered. Section~\ref{sec:mainresults} presents all our main results on the asymptotic behaviour of the mean and the tail of the sojourn time distribution under $\fb$. The results on the mean are then proven in Sections~\ref{sec:Tmean} and \ref{sec:TmeanGumbel}, whereas the results on the tail distribution are supported in Sections~\ref{sec:ToverET} and \ref{sec:Ttail}.

\section{Preliminaries}
\label{sec:preliminaries}
In this paper we consider a sequence of $\mg$ queues, indexed by $n$, where the $i$-th job requires $B_i$ units of service for all $n$. For convenience, we say that a job that requires $x$ units of service is a \textit{job of size $x$}. All $B_i$ are independently and identically distributed (i.i.d) random variables with cumulative distribution function (c.d.f.) $F(x) = \P(B_i\leq x)$ and finite mean $\E[B]$. We assume that $F(0)=0$, and denote $x_R:=\sup\{x\geq 0: F(x) < 1\}\leq \infty$. Jobs in the $n$-th queue arrive with rate $\l^{(n)}$, where $\l^{(n)}<1/\E[B_1]$ to ensure that the $n$-th system experiences work intensity $\r^{(n)} := \l^{(n)}\E[B_1] < 1$. For notational convenience, we let $B$ denote a random variable with c.d.f.\ $F$.

Let $\barF(x):=1-F(x)$ and $\Finv(y):=\inf\{x\geq 0: F(x)\geq y\}$ denote the complementary c.d.f.\ (c.c.d.f.) and the right-inverse of $F$ respectively. The random variable $B^*$ is defined by its c.d.f.\ $G(x):=\P(B^*\leq x) = \int_0^x \barF(t)/\E[B] \dd t$ and has $k$-th moment $\E[(B^*)^k]=\E[B^k]/(k\E[B])$. Since $\Ginv(y)$ is continuous and strictly increasing, its (right-)inverse $\Ginv(y)$ satisfies $\Ginv(G(x))=x$. Also, we recognize $h^*(x) := \frac{\barF(x)}{\E[B]\barG(x)}$ as the failure rate of $B^*$. One may deduce that $h^*(x)$ equals the reciprocal of the mean residual time; $h^*(x)=1/\E[B-x\mid B>x]$.

\subsection*{Foreground-Background scheduling policy}
Jobs are served according to the Foreground-Background ($\fb$) policy, meaning that at any moment in time, the server equally shares its capacity over all available jobs that have received the least amount of service thus far. First, we are interested in characteristics of the sojourn time $T^{(n)}_{\fb}$, defined as the duration of time that a generic job spends in the system. In order to analyse this, we consider an expression for the mean sojourn time of a generic job \textit{of size $x$}, $\E[T^{(n)}_{\fb}(x)]$, for which \citet{schrage1967queue} states that
\begin{equation}
  \E[T^{(n)}_{\fb}(x)] = \frac{x}{1-\r^{(n)}_x} + \frac{\E[W^{(n)}(x)]}{1-\r^{(n)}_x} = \frac{x}{1-\r^{(n)}_x} + \frac{\lambda^{(n)} m_2(x)}{2(1-\r^{(n)}_x)^2},
  \label{eq:Tx}
\end{equation}
%\begin{equation}
%  \rho_x := \lambda \E[B\wedge x] = \lambda\int_0^x t \dd F(t) + \lambda x\barF(x) = \lambda \int_0^x \barF(t) \dd t = \rho \P(B^*\leq x)
%  \label{eq:rhox}
%\end{equation}
where $\rho^{(n)}_x := \lambda^{(n)} \E[B\wedge x] = \rho \P(B^*\leq x)$ and $m_2(x) := \E[(B\wedge x)^2] = 2 \int_0^x t\barF(t)\dd t$ are functions of the first and second moments of $B\wedge x:=\min\{B,x\}$, and $W^{(n)}(x)$ is the steady-state waiting time in a $\mg/\fifo$ queue with arrival rate $\l^{(n)}$ and jobs of size $B_i\wedge x$. The intuition behind this result, is that a job $J_1$ of size $x$ experiences a system where all job sizes are truncated. Indeed, if another job $J_2$ of size $x+y,y>0$, has received at least $x$ service then $\fb$ will never dedicate its resources to job $J_2$ while job $J_1$ is incomplete. The mean sojourn time of job $J_1$ can now be salvaged from its own service requirement $x$, the truncated work already in the system upon arrival $W^{(n)}(x)$, and the rate $1-\r^{(n)}_x$ at which it is expected to be served, yielding~\eqref{eq:Tx}. As a consequence, the mean sojourn time $\E[T^{(n)}_{\fb}]$ of a generic job is given by
\begin{equation}
 \E[T^{(n)}_{\fb}] = \int_0^\infty \frac{x}{1-\r^{(n)}_x} \dd F(x) + \int_0^\infty \frac{\lambda^{(n)} m_2(x)}{2(1-\r^{(n)}_x)^2} \dd F(x).
 \label{eq:T}
\end{equation}

Second, we focus attention on the tail behaviour of $T^{(n)}_{\fb}$. Write $A \stackrel{d}{=} B$ if $\P(A\leq x) = \P(B\leq x)$ for all $x\in\R$ and let $\mathcal{L}_x(y)$ denote the time required by the server to empty the system given that job sizes are truncated to $B_i\wedge x$ and the current amount of work is $y$. The analysis of the tail behaviour is then facilitated by relation~(4.28) in \citet{kleinrock1976queueingcomputer}, stating 
\begin{equation}
  T_\fb^{(n)}(x) \stackrel{d}{=} \mathcal{L}_x(W_x^{(n)}+x).
  \label{eq:Tbusyperiod}
\end{equation}

For both the mean and tail behaviour of $T^{(n)}_{\fb}$, we take specific interest in systems that experience \textit{heavy traffic}, that is, systems where $\rho^{(n)}\uparrow 1$ as $n\rightarrow\infty$. In the current setting, this is equivalent to sequences $\l^{(n)}$ that converge to $1/\E[B]$. Most results in this paper make no assumptions on sequence $\l^{(n)}$, in which case we drop the superscript $n$ for notational convenience and just state $\rhotoone$.

The remainder of this section introduces some notation related to Matuszewska indices and extreme value theory.

\subsection*{Matuszewska indices}
\label{subsec:matuszewska}
%This section introduces the notion of the upper and lower Matuszewska index.
\begin{definition}
Suppose that $f(\cdot)$ is positive. 
\begin{itemize}
\item The \textit{upper Matuszewska index $\a(f)$} is the infimum of those $\a$ for which there exists a constant $C=C(\a)$ such that for each $\mu^*>1$,
  \begin{equation}
    \lim_{x\rightarrow\infty} f(\mu x) / f(x) \leq C \mu^\a
    \label{eq:upperMatuszewska}
  \end{equation}
  uniformly in $\mu\in[1,\mu^*]$ as $x\rightarrow\infty$.
\item The \textit{lower Matuszewska index $\beta(f)$} is the supremum of those $\beta$ for which there exists a constant $D=D(\beta)>0$ such that for each $\mu^*>1$,
  \begin{equation}
    \lim_{x\rightarrow\infty} f(\mu x) / f(x) \geq D \l^\beta
    \label{eq:lowerMatuszewska}
  \end{equation}
  uniformly in $\mu\in[1,\mu^*]$ as $x\rightarrow\infty$.
\end{itemize}
\end{definition}
One may note from the above definitions that $\b(f) = -\a(1/f)$ holds for any positive $f$. Intuitively, a function $f$ with upper and lower Matuszewska indices $\a(f)$ and $\b(f)$ is bounded between functions $D x^{\b(f)}$ and $C x^{\a(f)}$ for appropriate constants $C,D>0$. More accurately, however, $C$ and $D$ could be unbounded or vanishing functions of $x$. Of special interest is the class of functions that satisfy $\b(f)=\a(f)$.
\begin{definition}
A measurable function $f:\R_{\geq 0}\rightarrow \R_{\geq 0}$ is \textit{regularly varying at infinity with index $\a\in\R$} (written $f\in\RV_\a$) if for all $\mu >0$
\begin{equation}
  \lim_{x\rightarrow\infty} f(\mu x)/f(x) = \mu^\a.
  \label{eq:regularlyvarying}
\end{equation}
If~\eqref{eq:regularlyvarying} holds with $\a=0$, then $L$ is called \textit{slowly varying}. If~\eqref{eq:regularlyvarying} holds with $\a=-\infty$, then $L$ is called \textit{rapidly varying}.
\end{definition}

\subsection*{Extreme value theory} %Toe te voegen in proefschrift: MDA(\Phi_\a), MDA(\Psi_\a), voorbeelden
\label{subsec:extreme}
The following paragraphs introduce some notions and results from extreme value theory. %and are mainly based on Chapters~0 and 1 in \citet{resnick1987extreme}. 
The field of extreme value theory generally aims to assess the probability of an extreme event; however, for our purposes we restrict attention to the limiting distribution of $\max\{B_1,\ldots,B_n\}$. A key result on this functional is the Fisher-Tippett theorem:

\begin{theorem}[\citet{resnick1987extreme}, Proposition 0.3]
Let $(B_n)_{n\in\N}$ be a sequence of i.i.d.\ random variables and define $M_n:=\max\{B_1,\ldots,B_n\}$. If there exist norming sequences $c_n>0, d_n\in\R$ and some non-degenerate c.d.f\ $H$ such that
\begin{equation}
  c_n^{-1}(M_n - d_n) \stackrel{d}{\rightarrow} H,
  \label{eq:FisherTippett}
\end{equation}
then $H$ belongs to the type of one of the following three c.d.f.s:
\[
\begin{array}{rrcl}
\text{Fr{\'e}chet:} & \Phi_\a(x) & = & \left\{\begin{array}{ll} 0, &x\leq 0 \\ \exp\{-x^{-\a}\}, \quad &x>0\end{array}\right. \quad \a>0. \\
\text{Weibull:} & \Psi_\a(x) & = & \left\{\begin{array}{ll}\exp\{-(-x)^\a\}, \quad &x\leq 0 \\ 1, &x>0\end{array}\right. \quad \a>0. \\
\text{Gumbel:} & \Lambda(x) & = & \exp\{-e^{-x}\} \quad x\in\R. \\
\end{array}
\]
The three distributions above are referred to as the \textit{extreme value distributions}.
\end{theorem}

A c.d.f.\ $F$ is said to be in the \textit{maximum domain of attraction of $H$} if there exist norming sequences $c_n$ and $d_n$ such that \eqref{eq:FisherTippett} holds. In this case, we write $F\in \MDA(H)$. A large body of literature has identified conditions on $F$ such that $F\in\MDA(H)$. Excellent collections of such and related results can be found in \citet{embrechts1997modelling} and \citet{resnick1987extreme}. For reasons of convenience, we only mention a characterisation theorem for $\MDA(\Lambda)$:
\begin{theorem}[\citet{embrechts1997modelling}, Theorem 3.3.26] \label{thm:MDAGumbel}
The c.d.f. $F$ with right endpoint $x_R\leq \infty$ belongs to the maximum domain of attraction of $\Lambda$ if and only if there exists some $z<x_R$ such that $F$ has representation
\begin{equation}
  \barF(x) = c(x) \exp\left\{-\int_z^x \frac{g(t)}{f(t)} \dd t\right\}, \quad z<x<x_R,
  \label{eq:Gumbel}
\end{equation}
where $c$ and $g$ are measurable functions satisfying $c(x)\rightarrow c>0, g(t)\rightarrow 1$ as $x\uparrow x_R$, and $f(\cdot)$ is a positive, absolutely continuous function (with respect to the Lebesgue measure) with density $f'(x)$ having $\lim_{x\uparrow x_R} f'(x) = 0$.

If $F\in\MDA(\Lambda)$, then the norming constants can be chosen as $c_n = f(d_n)$ and $d_n=\Finv(1-n^{-1})$. A possible choice for the function $f(\cdot)$ is $f(\cdot) = 1/h^*(\cdot)$.
\end{theorem}
The function $f(\cdot)$ in the above definition is unique up to asymptotic equivalence. We refer to $f$ as the \textit{auxiliary function} of $\barF$. Also, we note the following property of $f(\cdot)$:

\begin{lemma}[\citet{resnick1987extreme}, Lemma~1.2] \label{lem:xfx}
Suppose that $f(\cdot)$ is an absolutely continuous auxiliary function with $f'(x)\rightarrow 0$ as $x\uparrow x_R$.
\begin{itemize}
\item If $x_R=\infty$, then $\lim_{x\rightarrow\infty} \frac{f(x)}{x} = 0$.
\item If $x_R<\infty$, then $\lim_{x\uparrow x_R} \frac{f(x)}{x_R-x} = 0$.
\end{itemize}
\end{lemma}

This section's final lemma shows that $G\in\MDA(\Lambda)$ whenever $F\in\MDA(\Lambda)$:
\begin{lemma} \label{lem:FandGareGumbel}
If $F\in\MDA(\Lambda)$, then $G\in\MDA(\Lambda)$ and any auxiliary function for $F$ is also an auxiliary function for $G$.
\end{lemma}
\begin{proof}
According to Theorem~3.3.27 in \citet{embrechts1997modelling}, $G\in\MDA(\Lambda)$ with auxiliary function $f(\cdot)$ if and only if 
$
  \lim_{x\uparrow x_R} \barG(x+t f(x))/\barG(x) = e^{-t}
$
for all $t\in\R$. It is straightforward to check that the above relation holds for any auxiliary function $f(\cdot)$ of $F$ by using l'H{\^o}pital and $\lim_{x\uparrow x_R} f'(x) = 0$.
%\begin{align*}
%  \lim_{x\uparrow x_R} \frac{\barG(x+t f(x))}{\barG(x)} 
%    &= \lim_{x\uparrow x_R} \left(1+t f'(x) \right)\frac{\barF(x+t f(x))}{\barF(x)}
%    = e^{-t}
%\end{align*}
\end{proof}

\subsection*{Asymptotic relations}
Let $f(\cdot)$ and $g(\cdot)$ denote two positive functions and $A$ and $B$ two random variables. We write $f \sim g$ if $\lim_{z\uparrow z^*} f(z)/g(z)=1$, where the appropriate limit $z\uparrow z^*$ depends on and should be clear from the context; it usually equals $x\uparrow x_R$ or $\rhotoone$. Similarly, we adopt the conventions $f=o(g)$ if $\limsup_{z\uparrow z^*} f(z)/g(z)=0$, $f=\O(g)$ if $\limsup_{z\uparrow z^*} f(z)/g(z) < \infty$ and $f=\Theta(g)$ if $0< \liminf_{z\uparrow z^*} f(z)/g(z) \leq \limsup_{z\uparrow z^*} f(z)/g(z) < \infty$. We write $A\leq_{st} B$ if the relation $P(A>x) \leq \P(B>x)$ is satisfied for all $x\in\R$.

Finally, the complementary error function is defined as $\erfc(x) := 2\pi^{-1/2}\int_x^\infty e^{-u^2}\dd u$.

\section{Main results and discussion}
\label{sec:mainresults}
This section presents and discusses our main results. Theorems~\ref{thm:ETMatuszewska} and \ref{thm:ETextreme} consider the asymptotic behaviour of the mean sojourn time $\E[T_\fb]$ for various classes of distributions. Theorem~\ref{thm:TmeanGumbel} connects the asymptotic behaviour of $\barF(\Ginv(\r))$ to the literature on extreme value theory. As a consequence, the expressions obtained in Theorem~\ref{thm:ETextreme} can be specified for many distributions in $\MDA(\Lambda)$. Theorem~\ref{thm:ToverET} shifts focus to the distribution of $T_\fb$ and states that the scaled sojourn time $T_\fb/\E[T_\fb]$ tends to zero in probability. Instead, Theorem~\ref{thm:Ttail} shows that a certain fraction of jobs experiences a sojourn time of order $(1-\r)^{-2}$. This result is achieved through the Laplace transform of the remaining sojourn time $T^*_\fb$, for which we give an integral presentation. The proofs of the theorems are postponed to later sections.

Recall that $\barF(\Ginv(\r)) = \E[B](1-\r)h^*(\Ginv(\r))$. Our first theorem presents the growth rate of $\E[T_\fb]$. 
\begin{theorem} \label{thm:ETMatuszewska}
Assume that either $x_R=\infty$ and $-\infty<\beta(\barF)\leq \a(\barF) < -2$, or that $x_R<\infty$ and $-\infty<\beta(\barF(x_R-(\cdot)^{-1}))\leq \a(\barF(x_R-(\cdot)^{-1})) < 0$. Then the relations
\begin{equation}
  \E[T_\fb] = \Theta\left(\frac{\barF(\Ginv(\r))}{(1-\r)^2}\right) = \Theta\left(\frac{h^*(\Ginv(\r))}{1-\r}\right)
\end{equation}
hold as $\rhotoone$, where $\lim_{\rhotoone} h^*(\Ginv(\r)) = 0$ if $x_R=\infty$ and $\lim_{\rhotoone} h^*(\Ginv(\r)) = \infty$ if $x_R<\infty$. 
Alternatively, assume $x_R=\infty$ and $\beta(\barF(x))> -2$. Then the relation
\begin{equation}
  \E[T_\fb] = \Theta\left(\log \frac{1}{1-\r}\right)
\end{equation}
holds as $\rhotoone$. 
\end{theorem}
Theorem~\ref{thm:ETMatuszewska} shows that the behaviour of $\E[T_\fb]$ is fundamentally different for $\a(\barF)<-2$ and $\beta(\barF(x))> -2$. In the first case, the variance of $F$ is bounded and therefore the expected remaining busy period duration is of order $\Theta((1-\r)^{-2})$. Our analysis roughly shows that all jobs of size $\Ginv(\r)$ and larger will remain in the system until the end of the busy period, and hence experience a sojourn time of order $\Theta((1-\r)^{-2})$. The theorem shows that, as the work intensity increases to unity, the contribution of these jobs to the average sojourn time determines the overall average sojourn time. 

The above argumentation does not apply in case $\beta(\barF(x))> -2$, since then the expected remaining busy period duration is infinite. It turns out that in this case the mean sojourn time of a large job of size $x$ is of the same order as the time that the job is in service, which has expectation $x/(1-\r_x)$. The result follows after integrating over the job size distribution.

Additionally, it can be shown that the statements in Theorem~\ref{thm:ETMatuszewska} also hold if $F\in\MDA(\Lambda)$, which is a special case of $\a(\barF)=\beta(\barF)=-\infty$ or $\a(\barF(x_R-(\cdot)^{-1}))=\beta(\barF(x_R-(\cdot)^{-1})=-\infty$. In this case, as well as in case $\barF(\cdot)$ or $\barF(x_R-(\cdot)^{-1})$ is regularly varying, one can show that $(1-\r)^2\E[T_\fb]/\barF(\Ginv(\r))$ converges. Theorem~\ref{thm:ETextreme} specifies Theorem~\ref{thm:ETMatuszewska} for the aforementioned cases, as well as for distributions with an atom in their endpoint.% (roughly speaking, the case $x_R<\infty, \beta(\barF(x_R-(\cdot)^{-1})=0$).

\begin{theorem} \label{thm:ETextreme} %\leavevmode \\
The following relations hold as $\rhotoone$:
\begin{enumerate}[(i)]
\item If $F\in\MDA(\Phi_\a), \a\in(1,2)$, then
  $%\begin{equation}
    \E[T_\fb] \sim \frac{\a}{2-\a}\E[B] \log\frac{1}{1-\r}.
  $ %\end{equation}
\item If $F\in\MDA(H)$, then
  $%\begin{equation}
    \E[T_\fb] \sim  \frac{r(H) \E[B^*] \barF(\Ginv(\r))}{(1-\r)^2} = \frac{r(H) \E[B^2] h^*(\Ginv(\r))}{2(1-\r)}
  $ %\end{equation}
  where
  \begin{equation}
    r(H) = \left\{\begin{array}{ll}
      \frac{\pi/(\a-1)}{\sin(\pi/(\a-1))} \frac{\a}{\a-1} \quad & \text{ if } H=\Phi_\a, \a>2, \\
      1 & \text{ if } H=\Lambda, \text{ and } \\
      \frac{\pi/(\a+1)}{\sin(\pi/(\a+1))} \frac{\a}{\a+1} & \text{ if } H=\Psi_\a, \a>0.
    \end{array}\right.
    \label{eq:rH}
  \end{equation}
  Additionally, if $H=\Phi_\a, \a>2,$ then $\lim_{\rhotoone} h^*(\Ginv(\r)) = 0$, whereas if either $H=\Lambda$ and $x_R<\infty$ or if $H=\Psi_\a,\a>0,$ then $\lim_{\rhotoone} h^*(\Ginv(\r)) = \infty$.
\item If $F$ has an atom in $x_R<\infty$, say $\lim_{\d\downarrow 0} \barF(x_R-\d)=p>0$, then
  $%\begin{equation}
  \E[T_\fb] \sim \frac{p\E[B^*]}{(1-\r)^2}.
  $ %\end{equation}
\end{enumerate}
\end{theorem}

The expressions in Theorems~\ref{thm:ETMatuszewska} and \ref{thm:ETextreme} give insight into the asymptotic behaviour of $\E[T_\fb]$. The following corollary shows that the asymptotic expressions above may be specified further if the job sizes are Pareto distributed. This extends the result by \citet{bansal2006handling}, who derived the growth factor of $\E[T_\fb]$ but not the exact asymptotics.
\begin{corollary} \label{cor:ETPareto}
Assume $\barF(x) = (x/x_L)^{-\a}, x\geq x_L$. Then the relations
\begin{equation}
  \E[T_\fb] \sim \left\{\begin{array}{ll}
    \frac{\a}{2-\a}\E[B] \log\frac{1}{1-\r} & \text{if } \a\in(1,2), \\
    \frac{\pi/(\a-1)}{2\sin(\pi/(\a-1))}  \frac{\E[B^2]\a^{\frac{\a}{\a-1}}}{x_L(1-\r)^{\frac{\a-2}{\a-1}}} \quad & \text{if } \a\in(2,\infty),
  \end{array}\right.
\end{equation}
hold as $\rhotoone$.
\end{corollary}
\begin{proof}
One may derive that $\barG(x) = \frac{1}{\a}\left(\frac{x}{x_L}\right)^{1-\a}$ for $x\geq x_L$ and sequentially that $h^*(x)=\frac{\a-1}{x}$ for $x\geq x_L$ and $\Ginv(\r) = x_L(\a(1-\r))^{\frac{-1}{\a-1}}$ for $\r\geq 1-1/\a$. The result then follows from Theorem~%\ref{thm:ETFrechet}.
\ref{thm:ETextreme}.
\end{proof}

Corollary~\ref{cor:ETPareto} exemplifies that the asymptotic growth of $\E[T_\fb]$ may be specified in some cases. However, it is often non-trivial to analyse the behaviour of $\barF(\Ginv(\r))$ or equivalently $h^*(\Ginv(\r))$. Theorem~\ref{thm:TmeanGumbel} aims to overcome this problem if $F\in\MDA(\Lambda)$ by presenting a relation between $h^*(\Ginv(\r))$ and norming constants $c_n$ of $F$, which can often be found in the large body of literature on extreme value theory. 

\begin{theorem} \label{thm:TmeanGumbel}
  Assume $F\in \MDA(\Lambda)$ and $x_R=\infty$, and let $c_n$, $d_n$ be such that $c_n^{-1}(\barF^n-d_n) \stackrel{d}{\rightarrow} \Lambda$. Define $\l^{(n)} = (1-n^{-1})/\E[B]$ so that $\r^{(n)}=1-n^{-1}$.
  \begin{enumerate}[(i)]
  \item If there exists $\a>0$ and a slowly varying function $l(x)$ such that $-\log \barF(x)\sim l(x) x^\a$ as $x\rightarrow\infty$, then $h^*(x)\sim \a l(x) x^{\a-1}$ if and only if
    \begin{equation}
      \inf_{\l\downarrow 1}\liminf_{x\rightarrow\infty}\inf_{t\in[1,\l]}\{\log h^*(tx) - \log h^*(x)\} \geq 0.
      \label{eq:tauberian}
    \end{equation}
    If \eqref{eq:tauberian} holds, then $\E[T_\fb^{(n)}] \sim \frac{\E[B^2]}{2(1-\r^{(n)})c_n}$ as $n\rightarrow\infty$.
  \item If there exists a function $l(x):[0,\infty)\rightarrow \R, \liminf_{x\rightarrow \infty} l(x) >1$ such that for all $\l>0$
    \begin{equation}
      \lim_{x\rightarrow\infty} \frac{-\log\barF(\l x)+\log\barF(x)}{l(x)} = \log(\l)
    \end{equation}
    and $L=\lim_{x\rightarrow \infty} \frac{\log(x)}{l(x)}$ exists in $[0,\infty]$, then $\lim_{n\rightarrow\infty} \frac{2(1-\r^{(n)})c_n}{\E[B^2]} \E[T_\fb^{(n)}] = e^{-L}$.
  \end{enumerate}
  The same results hold if $x_R<\infty$, provided that the $\barF(\cdot)$ and $h^*(\cdot)$ in (i) and (ii) are replaced by $\barF(x_R-\frac{1}{\cdot})$ and $h^*(x_R-\frac{1}{\cdot})/(\cdot)^2$, respectively.
\end{theorem}
\remark{Condition~\eqref{eq:tauberian} in part (i) of Theorem~\ref{thm:TmeanGumbel} is a \textit{Tauberian condition}, and origins from Theorem~1.7.5 in \citet{bingham1989regular}. A Tauberian theorem makes assumptions on a transformed function (here $h^*$), and uses these assumptions to deduce the asymptotic behaviour of that transform. It is non-restrictive in the sense the result in the theorem holds, i.e.\ if $h^*(x) \sim\a l(x) x^{\a-1}$, then obviously condition~\eqref{eq:tauberian} is met and therefore the Tauberian condition does not restrict the class of functions $F$ to which the theorem applies. However, condition~\eqref{eq:tauberian} is necessary for the result to hold. The interested reader is referred to Section~XIII.5 in \citet{feller1968introduction} and Section~1.7 in \citet{bingham1989regular}.}
%\alert{Unsure how to deal with the (non-restrictive) Tauberian condition in (i)? Beirland et al.\ need this in order to apply Theorem 1.7.5 in Bingham et al., where condition 1.7.10'' is the condition at hand.} \nts{If we assume the Tauberian condition holds (and thereby restrict the possible outcomes of the theorem), then we know that $h^*(x)$ behaves like $\a l(x) x^{\a-1}$ asymptotically. If we assume that $h^*$ does not satisfy the Tauberian condition, then there might exist an $\tilde{h}^*(x)$ that has different asymptotic behaviour. May we just pick either of these functions to work with? Seems to me that both should correspond to the `real' function $h^*$, meaning that all $\tilde{h}^*$ must be asymptotically equivalent to $\a l(x) x^{\a-1}$, which is impossible and therefore no such $\tilde{h}^*$ can exist.. Must be a flaw in this reasoning.}

Theorem~\ref{thm:MDAGumbel} implies that $c_n \sim 1/h^*(\Ginv(1-n^{-1}))$ for many distributions in $\MDA(\Lambda)$. As $c_n$ may be chosen as $1/h^*(\Finv(1-n^{-1}))$, Theorem~\ref{thm:TmeanGumbel} implicitly states conditions under which $\lim_{n\rightarrow \infty} h^*(\Ginv(1-n^{-1}))/h^*(\Finv(1-n^{-1})) = \lim_{y\uparrow 1} (1-y)^{-2} \barF(\Ginv(y))\barG(\Finv(y))$ exists, and exploits this limit to write $\E[T^{(n)}_{\fb}]$ as function of $c_n$ rather than of $h^*(\Ginv(1-n^{-1}))$. As to illustrate the implications of Theorem~\ref{thm:TmeanGumbel}, the exact asymptotic behaviour of several well-known distributions is presented in Table~\ref{tab:Gumbel}.
\begin{sidewaystable}
%\begin{landscape}
%\begin{table}[p]
\[\begin{array}{|lll|c|l|}
\hline
\text{Distribution} & \text{c.c.d.f.\ $\barF$ or p.d.f.\ $F'$} && L & \E[T_\fb] \sim \\ %\frac{\E[B^2] h^*(\Ginv(\r))}{2(1-\r)}
\hline
\text{Exponential-like} & \barF(x) \sim K e^{-\mu x} & K,\mu>0 & - & \frac{\E[B^2]\mu}{2(1-\r)} \\
\text{Weibull-like} &\barF(x) \sim K x^\a e^{-\mu x^\beta} & K,\mu,\beta>0, \a\in\R & - & \frac{\beta \mu^{1/\beta} \E[B^2]}{2(1-\r)\log\left(\frac{1}{1-\r}\right)^{1/\beta-1}} \\
\text{Gamma} &F'(x) = \frac{\beta^\a}{\Gamma(\a)}x^{\a-1}e^{-\beta x} & \a,\b>0 & - & \frac{\E[B^2] \beta}{2(1-\r)} \\
\text{Normal} &F'(x) = \frac{1}{\sqrt{2\pi}} e^{-x^2/2} & & - & \frac{\E[B^2] \log\left(\frac{1}{1-\r}\right)^{1/2}}{\sqrt{2} (1-\r)} \\
\text{Lognormal} & F'(x) = \frac{1}{\sqrt{2\pi}\sigma x}e^{-(\log(x)-\mu)^2/(2\sigma^2)} & \sigma>0, \mu\in\R & \sigma^2 & \frac{e^{-\sigma^2} \E[B^2] \log\left(\frac{1}{1-\r}\right)^{1/2} }{\sigma \sqrt{2}(1-\r)\exp\left[\mu + \sigma\left(\sqrt{2\log\left(\frac{1}{1-\r}\right)}-\frac{\log(4\pi) + \log\log(1/(1-\r))}{2 \sqrt{2\log(1/(1-\r))}}\right)\right]} \\
\text{Finite Exponential} & \barF(x) = K e^{-\frac{\mu}{x_R-x}} & K,\mu>0, x<x_R & - & \frac{\E[B^2] \log\left(\frac{1}{1-\r}\right)^2}{2\mu(1-\r)} \\
\text{Benktander-I} & \barF(x)=\left(1+2\frac{\beta}{\alpha}\log(x)\right) & \a,\beta>0, x>1 & \frac{1}{2\beta} & \frac{e^{-\frac{1}{2\beta}}\E[B^2] \sqrt{\beta\log\left(\frac{1}{1-\r}\right)}}{(1-\r)\exp\left[-\frac{\a+1}{\beta}+\sqrt{\frac{\log\left(\frac{1}{1-\r}\right)}{\beta}}\right]} \\
 & \qquad \times e^{-(\beta\log(x)^2+(\a+1)\log(x))} &&& \\
\text{Benktander-II} & \barF(x) = x^{-(1-\beta)} e^{-\frac{\a}{\beta}(x^\beta-1)} & \a>0, 0<\beta<1, x>1 & - & \frac{\a^{1/\beta}\E[B^2]}{2\beta^{1/\beta-1}(1-\r)\log\left(\frac{1}{1-\r}\right)^{1/\beta-1}} \\
\hline
\end{array}\]
\caption{Asymptotic expressions for the mean sojourn time for several well-known distributions in $\MDA(\Lambda)$, characterized by either their tail distribution or their probability density function (p.d.f.). These expressions follow from Table~3.4.4 in \citet{embrechts1997modelling} through Theorem~\ref{thm:TmeanGumbel}, where it is assumed that relation~\eqref{eq:tauberian} holds.}
\label{tab:Gumbel}
%\end{table}
%\end{landscape}
\end{sidewaystable}

We take a brief moment to compare the asymptotic mean sojourn time under $\fb$ to that under $\srpt$ in $\mg$ models. Clearly, $\fb$ can perform no better than $\srpt$ due to $\srpt$'s optimality \citep{schrage1968letter}. The ratio of their respective mean sojourn time is shown to be unbounded if the job sizes are Exponentially distributed or if the job size distribution has finite support %\citep{bansal2005average, kleinrock1976queueingcomputer, lin2011heavy}, Lemma~9.13 in Nuyens' dissertation
\citep{kleinrock1976queueingcomputer, bansal2005average, nuyens2008foreground, lin2011heavy}, but bounded if the job sizes are Pareto distributed \citep{bansal2006handling, lin2011heavy}. To the best of the authors' knowledge, no results of this nature are known if job sizes are Weibull distributed.

The following corollary specifies the asymptotic advantage of $\srpt$ over $\fb$ if the job sizes are Pareto distributed, and presents the first such results for Weibull distributed job sizes. Its statements follow directly from Corollaries~1 and 2 in \citet{lin2011heavy} and the earlier results in this section.
\begin{corollary}
\label{cor:FBvsSRPT}
The following relations hold as $\rhotoone$:
\begin{enumerate}[(i)]
\item If $\barF(x) = (x/x_L)^{-\a}, x\geq x_L>0$ and $\a\in(1,2)$, then $\E[T_\fb]/\E[T_{\srpt}] \sim \a^2$.
\item If $\barF(x) = (x/x_L)^{-\a}, x\geq x_L>0$ and $\a>2$, then $\E[T_\fb]/\E[T_{\srpt}] \sim \a^{\frac{\a}{\a-1}}$.
\item If $\barF(x) = e^{-\mu x^\beta}, x\geq 0$ and $\beta>0$, then $\E[T_\fb]/\E[T_{\srpt}] \sim \beta \log\left(\frac{1}{1-\r}\right)$.
\end{enumerate}
\end{corollary}

Now that the asymptotic behaviour of the mean sojourn time under $\fb$ has been quantified, it is natural to investigate more complex characteristics. One such characteristic is the behaviour of the tail of the sojourn time distribution, where one usually starts by analysing the distribution of the sojourn time normalized by its mean, $T_\fb/\E[T_\fb]$. The following theorem indicates that this random variable converges to zero in probability, meaning that almost every job experiences a sojourn time that is significantly shorter than the mean sojourn time as $\rhotoone$:
\begin{theorem} \label{thm:ToverET}
If either
\begin{itemize}
\item $x_R=\infty$ and either $\beta(\barF)>-2$ or $-\infty<\beta(\barF)\leq \a(\barF) < -2$, or
\item $x_R<\infty$ and $-\infty<\beta(\barF(x_R-(\cdot)^{-1}))\leq \a(\barF(x_R-(\cdot)^{-1})) < 0$, or
\item $F\in\MDA(\Lambda)$,
\end{itemize}
then $\frac{T_\fb}{\E[T_\fb]}\stackrel{p}{\rightarrow} 0$ as $\rhotoone$.
\end{theorem}

Theorem~\ref{thm:ToverET} indicates that a decreasing fraction of jobs experiences a sojourn time of at least duration $\E[T_\fb]$. Our final main result aims to specify both the size of this fraction, and the growth factor of the associated jobs' sojourn time.

The intuition from the proof of Theorem~\ref{thm:ETMatuszewska} suggests that $T_\fb$ scales as $(1-\r)^{-2}$, but only for jobs of size at least $\Ginv(\r)$. This makes it conceivable that the scaled probability $\P((1-\r)^2 T_\fb > y)/\barF(\Ginv(\r))$ may be of $\Theta(1)$ as $\rhotoone$. Theorem~\ref{thm:Ttail} confirms this hypothesis, and additionally shows that the residual sojourn time $T_\fb^*$ with density $\P(T_\fb>x)/\E[T_\fb]$ scales as $(1-\r)^{-2}$.
\begin{theorem} \label{thm:Ttail}
Assume $F\in\MDA(H)$, where $H$ is an extreme value distributions with finite $(2+\e)$-th moment for some $\e>0$. Let $r(H)$ be as in relation~\eqref{eq:rH}. Then $(1-\r)^2T^*_\fb$ converges to a non-degenerate random variable with monotone density $g^*$ as $\rhotoone$, and
%\begin{equation}
%  \lim_{\rhotoone} \E[e^{-q(1-\r)^2T^*_\fb}] = \int_0^\infty e^{-qt} g^*(t) \dd t
%\end{equation}
%and
\begin{equation}
  \lim_{\rhotoone} \frac{\P((1-\r)^2T_\fb>y)}{r(H)\E[B^*]\barF(\Ginv(\r))} = g^*(y)
\end{equation}
almost everywhere. Here,
\begin{align}
  g^*(t) &= \int_0^1 r(H)^{-1} 8 \nu g(t,\nu) \left(\frac{1-\nu}{\nu}\right)^{p(H)} \dd \nu, \\
  g(t,\nu) &= \frac{e^{-\frac{t}{4\E[B^*]\nu^2}}}{4\E[B^*]\nu^2} \left(\frac{\sqrt{t}}{\nu \sqrt{\pi \E[B^*]}} - \frac{t}{2\E[B^*]\nu^2}e^{\frac{t}{4\E[B^*]\nu^2}} \erfc\left(\frac{1}{2\nu}\sqrt{\frac{t}{\E[B^*]}}\right)\right),
\end{align}
and $p(H)=\frac{\a}{\a-1}$ if $H=\Phi_\a,\a>2$; $p(H)=1$ if $H=\Lambda$ and $p(H)=\frac{\a}{\a+1}$ if $H=\Psi_\a,\a>0$.
\end{theorem}

%\nts{Expect that $\frac{\P((1-\r)^2T_\fb>y)}{\barF(x_\r)}$ also converges if $F\in\MDA(\Phi_\a),\a\in(1,2)$, where $x_\r$ is such that $x_\r/\barG(x_\r) = \log(1/(1-\r))$. Probably takes too much time to verify, so only do if rest of thesis is finished.}

All theorems presented in this section are now proven in order. First, Theorems~\ref{thm:ETMatuszewska} and \ref{thm:ETextreme} are proven in Section~\ref{sec:Tmean}. Then, Theorem~\ref{thm:TmeanGumbel} is justified in Section~\ref{sec:TmeanGumbel}. Finally, Sections~\ref{sec:ToverET} and \ref{sec:Ttail} respectively validate Theorems~\ref{thm:ToverET} and \ref{thm:Ttail}.

\section{Asymptotic behaviour of the mean sojourn time} 
\label{sec:Tmean}
In this section, we prove Theorems~\ref{thm:ETMatuszewska} and \ref{thm:ETextreme} in order. The intuition behind the theorems is that jobs of size $x$ can only be completed once the server has finished processing of all jobs of size at most $x$. Additionally, jobs of size $x$ experience a system with job sizes $B_i\wedge x$ since no job will receive more than $x$ units of processing as long as there are size $x$ jobs in the system. One thus expects all jobs of size $x$ to stay in the system for the duration of a remaining busy period in the truncated system, which is expected to last for $%\Theta(\E[(B\wedge x)^2]/(1-\r G(x))^2)=
\Theta(\E[(B\wedge x)^2]/(1-\r_x)^2)$ time. 

Now, if $\E[B^2]<\infty$ and $x_\r^\nu$ is such that $(1-\r)/(1-\r_{x_\r^\nu})=\nu\in(1-\r,1)$, then one can see from \eqref{eq:Tx} that
\begin{equation}
  (1-\r)^2\E[T_\fb(x_\r^\nu)] = \nu(1-\r)x_\r^\nu + \nu^2 \frac{\l m_2(x_\r^\nu)}{2}.
  \label{eq:Txscaled}
\end{equation}
It turns out that the asymptotic behaviour of $(1-\r)^2\E[T_\fb]$ is now determined by the fraction of jobs for which $\nu$ takes values away from zero.

%If instead $\E[B^2]=\infty$, then the growth rate of the second term in \eqref{eq:Tx} is only slightly higher than the decay rate of the fraction of jobs for which $(1-\r)/(1-\r_{x_\r^\nu})$ is sufficiently large. It turns out that the sojourn time is of the same order as the time that a job receives service, which is of order $\Theta(x/(1-\r_x))$.
If instead $\E[B^2]=\infty$, then it will be shown that the growth rate of the second term in \eqref{eq:Tx} is bounded by the growth rate of $x\barG(x)$. It turns out that the sojourn time is of the same order as the time that a job receives service, which is of order $\Theta(x/(1-\r_x))$.

Both theorems follow after integrating $\E[T_\fb(x)]$ over all possible values of $x$, as shown in \eqref{eq:T}. By integrating by parts, we find that the first integral in \eqref{eq:T} can be rewritten as
\begin{align*}
  \int_0^\infty \frac{x}{1-\r_x} \dd F(x) 
%  &= \left[-\frac{x\barF(x)}{1-\r_x}\right]_{x=0}^\infty + \int_0^\infty \barF(x) \dd \frac{x}{1-\r_x} \\
  &= \int_0^\infty \frac{\barF(x)}{1-\r_x} \dd x + \l \int_0^\infty \frac{x \barF(x)^2}{(1-\r_x)^2}\dd x \\
%  &= \frac{1}{\l} \left[-\log(1-\r_x)\right]_{x=0}^\infty + \l \int_0^\infty \frac{x \barF(x)^2}{(1-\r_x)^2}\dd x \\
  &= \frac{1}{\l} \log\frac{1}{1-\r} + \l \int_0^\infty \frac{x \barF(x)^2}{(1-\r_x)^2}\dd x.
\end{align*}
Similarly, the second integral can be rewritten as
\begin{align*}
  \int_0^\infty \frac{\lambda m_2(x)}{2(1-\r_x)^2} \dd F(x) 
%    &= \left[-\frac{\lambda m_2(x) \barF(x)}{2(1-\r_x)^2}\right]_{x=0}^\infty + \int_0^\infty \barF(x) \dd \frac{\lambda m_2(x)}{2(1-\r_x)^2} \\
    &= \lambda \int_0^\infty \frac{x\barF(x)^2}{(1-\r_x)^2} \dd x + \lambda^2 \int_0^\infty \frac{m_2(x) \barF(x)^2}{(1-\r_x)^3} \dd x,
\end{align*}
and therefore
\begin{align*}
  \E[T_\fb] 
    &= \frac{1}{\l} \log\frac{1}{1-\r} + 2\l \int_0^\infty \frac{x \barF(x)^2}{(1-\r_x)^2}\dd x + \lambda^2 \int_0^\infty \frac{m_2(x) \barF(x)^2}{(1-\r_x)^3} \dd x \\
    &= \frac{\E[B]}{\r} \log\frac{1}{1-\r} + 2\r \int_0^\infty \frac{x \barF(x)}{(1-\r_x)^2}\dd G(x) + \frac{\r^2}{\E[B]} \int_0^\infty \frac{m_2(x) \barF(x)}{(1-\r_x)^3} \dd G(x). \numberthis \label{eq:Tmain}
\end{align*}

We will now derive Theorems~\ref{thm:ETMatuszewska} and \ref{thm:ETextreme} from this relation.

\subsection{General Matuszewska indices}
This section proves Theorem~\ref{thm:ETMatuszewska}. Relation~\eqref{eq:Tmain} will be analysed separately for the cases $-\infty<\beta(\barF) \leq \a(\barF) < -2$ and $-2<\beta(\barF) \leq \a(\barF) < 1$, which will be referred to as the finite and the infinite variance case, respectively. The finite variance case also considers $-\infty<\beta(\barF(x_R-(\cdot)^{-1})$. Note that we always have $\beta(\barF(x_R-(\cdot)^{-1}))\leq \a(\barF(x_R-(\cdot)^{-1})) \leq 0$ since $\barF(x_R-(\cdot)^{-1})$ is non-increasing. Prior to further analysis, however, we introduce several results that will facilitate the analysis. 

\begin{lemma} \label{lem:onesidedMatuszewskaproduct}
Let $f_1(\cdot), f_2(\cdot)$ be positive functions.
\begin{enumerate}[(i)]
\item If $\a(f_1),\a(f_2)<\infty$, then $\a(f_1\cdot f_2)\leq \a(f_1)+\a(f_2)$ and, assuming that $f_1$ is non-decreasing, $\a(f_1\circ f_2)\leq \a(f_1)\cdot \a(f_2)$.
\item If $\beta(f_1),\beta(f_2)>-\infty$, then $\beta(f_1\cdot f_2)\geq \beta(f_1)+\beta(f_2)$ and, assuming that $f_1$ is non-increasing, $\beta(f_1\circ f_2)\geq \beta(f_1)\cdot \beta(f_2)$.
\end{enumerate}
\end{lemma}
\begin{lemma} \label{lem:onesidedMatuszewskavanish}
Let $f$ be positive. If $\a(f)<0$, then $\lim_{\xtoinfty} f(x) = 0$.
\end{lemma}
\begin{lemma}[\citet{bingham1989regular}, Theorem~2.6.1] \label{lem:onesidedMatuszewskaPI}
Let $f$ be positive and locally integrable on $[X,\infty)$. Let $g(x):= \int_X^x f(t)/t \dd t$. If $\beta(f)>0$, then $\liminf_{\xtoinfty} f(x) / g(x) >0$. %and $\limsup_{\xtoinfty} f(x) / g(x) \geq \a(f)$.
\end{lemma}
\begin{lemma}[\citet{bingham1989regular}, Theorem~2.6.3] \label{lem:onesidedMatuszewskaPD}
Let $f$ be positive and measurable. Let $g(x):= \int_x^\infty f(t)/t \dd t$. 
\begin{enumerate}[(i)]
\item If $\a(f)<0$, then $g(x)<\infty$ for all large $x$.
\item If $\b(f)>-\infty$, then $\limsup_{\xtoinfty} f(x)/g(x) < \infty$.
\end{enumerate}
\end{lemma}
\begin{lemma} \label{lem:barGMatuszewska}
If $x_R=\infty$, then $\a(\barG) \leq \a(\barF) + 1$ and $\beta(\barG) \geq \beta(\barF) + 1$. Alternatively, if $x_R<\infty$, then $\a(\barG(x_R-(\cdot)^{-1})) \leq \a(\barF(x_R-(\cdot)^{-1})) - 1$ and $\beta(\barG(x_R-(\cdot)^{-1})) \geq \beta(\barF(x_R-(\cdot)^{-1})) - 1$.
\end{lemma}
\begin{lemma}\label{lem:inverseMatuszewska2}
If $x_R=\infty$ and $\beta(\barF) >-\infty$, then $\beta(\Ginv(1-(\cdot)^{-1})) \geq -1/(\beta(\barF)+1)$ and $\alpha(\Ginv(1-(\cdot)^{-1})) \leq -1/(\alpha(\barF)+1)$. Alternatively, if $x_R<\infty$ and $\beta(\barF(x_R-(\cdot)^{-1})) >-\infty$, then $\beta(\Ginv(1-(\cdot)^{-1})) \geq -1/(\beta(\barF(x_R-(\cdot)^{-1}))-1)$ and $\alpha(\Ginv(1-(\cdot)^{-1})) \leq -1/(\alpha(\barF(x_R-(\cdot)^{-1}))-1)$.
\end{lemma}
\begin{corollary}\label{cor:FGinvMatuszewska}
If $x_R=\infty$ and $\beta(\barF)>-\infty$, then $\beta(\barF(\Ginv(1-(\cdot)^{-1}))) \geq \frac{-\beta(\barF)}{\beta(\barF)+1}$ and $\alpha(\barF(\Ginv(1-(\cdot)^{-1}))) \leq \frac{-\a(\barF)}{\a(\barF)+1}$. Alternatively, if $x_R<\infty$ and $\beta(\barF(x_R-(\cdot)^{-1}))>-\infty$, then $\beta(\barF(\Ginv(1-(\cdot)^{-1}))) \geq \frac{-\beta(\barF(x_R-(\cdot)^{-1}))}{\beta(\barF(x_R-(\cdot)^{-1}))-1}$ and $\alpha(\barF(\Ginv(1-(\cdot)^{-1}))) \leq \frac{-\a(\barF(x_R-(\cdot)^{-1}))}{\a(\barF(x_R-(\cdot)^{-1}))-1}$.
\end{corollary}
Lemma~\ref{lem:onesidedMatuszewskaproduct} states some closure properties of Matuszewska indices. Lemma~\ref{lem:onesidedMatuszewskavanish} gives a sufficient condition for $f$ to vanish. Lemmas~\ref{lem:onesidedMatuszewskaPI} and \ref{lem:onesidedMatuszewskaPD} state helpful results on the asymptotic behaviour of the ratio between a function and certain integrals over this function, depending on its Matuszewska indices. Lemmas~\ref{lem:barGMatuszewska} and \ref{lem:inverseMatuszewska2} and Corollary~\ref{cor:FGinvMatuszewska} specify and append the earlier lemmas by giving bounds on the Matuszewska indices of $\Ginv$ and the composition of $\barF$ and $\Ginv$. The proofs of Lemmas~\ref{lem:onesidedMatuszewskaproduct}, \ref{lem:onesidedMatuszewskavanish}, \ref{lem:barGMatuszewska} and \ref{lem:inverseMatuszewska2}, along with several additional results, are postponed to Appendix~\ref{app:preliminaries}. Corollary~\ref{cor:FGinvMatuszewska} follows immediately from Lemmas~\ref{lem:onesidedMatuszewskaproduct} and \ref{lem:inverseMatuszewska2}.

\subsubsection{Finite variance} \label{subsubsec:finitevariance}
In this section, we assume either $x_R=\infty$ and $-\infty<\beta(\barF)\leq \a(\barF)<-2$, or $x_R<\infty$ and $\beta(\barF(x_R-(\cdot)^{-1})) >-\infty$. If $x_R=\infty$, then $\a((\cdot)^2 \barF(\cdot))<0$ and thus $\E[B^2] = 2\int_0^\infty t \barF(t) \dd t < \infty$ by Lemma~\ref{lem:onesidedMatuszewskaPD}(i); if $x_R<\infty$ then clearly $\E[B^2]<\infty$.

Noting that $\Ginv$ is a continuous, strictly increasing function, it follows that the function $x_\r^\nu:=\Ginv\left(1-\frac{1-\r}{\r}\frac{1-\nu}{\nu}\right)$ is well-defined for all $\nu\in(1-\r,1)$. For this choice of $x_\r^\nu$, we have $\frac{1-\r}{1-\r_{x_\r^\nu}}=\nu$ and $\frac{\dd G(x_\r^\nu)}{\dd \nu} = \frac{1-\r}{\r}\frac{1}{\nu^2}$, and therefore relation \eqref{eq:Tmain} becomes
\begin{align*}
  (1-\r)^2\E[T_\fb] \hspace{-28pt}&\hspace{28pt} 
    = \frac{\E[B](1-\r)^2}{\r} \log\frac{1}{1-\r} 
      + 2\r \int_0^\infty \left(\frac{1-\r}{1-\r_x}\right)^2 x \barF(x)\dd G(x) \\
        &\hspace{28pt} \qquad + \frac{\r^2}{\E[B]} \int_0^\infty \left(\frac{1-\r}{1-\r_x}\right)^3 \frac{m_2(x) \barF(x)}{1-\r} \dd G(x) \\
    &= \frac{\E[B](1-\r)^2}{\r} \log\frac{1}{1-\r} \\
      &\qquad + 2(1-\r) \int_{1-\r}^1 \Ginv\left(1-\frac{1-\r}{\r}\frac{1-\nu}{\nu}\right) \barF\left(\Ginv\left(1-\frac{1-\r}{\r}\frac{1-\nu}{\nu}\right)\right) \dd \nu \\
        &\qquad \qquad + \frac{\r}{\E[B]} \int_{1-\r}^1 \nu \, m_2\left(\Ginv\left(1-\frac{1-\r}{\r}\frac{1-\nu}{\nu}\right)\right) \barF\left(\Ginv\left(1-\frac{1-\r}{\r}\frac{1-\nu}{\nu}\right)\right) \dd \nu.
\end{align*}
Dividing both sides by $\barF(\Ginv(\r))$ yields
\begin{align*}
  \frac{(1-\r)^2\E[T_\fb]}{\barF(\Ginv(\r))} &
    = \frac{\E[B](1-\r)^2}{\r\barF(\Ginv(\r))} \log\frac{1}{1-\r} \\
      &\qquad + \frac{2(1-\r)}{\barF(\Ginv(\r))} \int_{1-\r}^1 \Ginv\left(1-\frac{1-\r}{\r}\frac{1-\nu}{\nu}\right) \barF\left(\Ginv\left(1-\frac{1-\r}{\r}\frac{1-\nu}{\nu}\right)\right) \dd \nu \\
        &\hspace{-4pt} \qquad \qquad + \frac{\r}{\E[B]} \int_{1-\r}^1 \nu \, m_2\left(\Ginv\left(1-\frac{1-\r}{\r}\frac{1-\nu}{\nu}\right)\right) \frac{\barF\left(\Ginv\left(1-\frac{1-\r}{\r}\frac{1-\nu}{\nu}\right)\right)}{\barF(\Ginv(\r))} \dd \nu \\
    &= \mathrm{I}(\r) + \mathrm{II}(\r) + \mathrm{III}(\r) \numberthis \label{eq:meanMatuszewskadecomposed}
\end{align*}
We will show that $\mathrm{I}(\r) + \mathrm{II}(\r) = o(1)$ and $\mathrm{III}(\r) = \Theta(1)$. Assume $x_R=\infty$. Then, %then $-\infty<\beta(\barF)\leq \a(\barF) < -2$, 
by Lemma~\ref{lem:onesidedMatuszewskaproduct} and Corollary~\ref{cor:FGinvMatuszewska} we find that 
\begin{align*}
  \a(\mathrm{I}(1-(\cdot)^{-1}) 
    &\leq \a((\cdot)^{-2}) + \a(1/\barF(\Ginv(1-(\cdot)^{-1}))) + \a(\log (\cdot)) \\
    &= -2 -\beta(\barF(\Ginv(1-(\cdot)^{-1}))) + 0 
    \leq -2 + \frac{\beta(\barF)}{\beta(\barF)+1}
    < 0, \numberthis \label{eq:alphaIinfinite}
\end{align*}
and consequently $\mathrm{I}(\r) = o(1)$ as $\rhotoone$ by Lemma~\ref{lem:onesidedMatuszewskavanish}.
%Similarly, if $x_R<\infty$ then 
%\begin{align*}
%  \a(\mathrm{I}(1-(\cdot)^{-1}) 
%    %&= \a((\cdot)^{-2}) + \a(1/\barF(\Ginv(1-(\cdot)^{-1}))) + \a(\log (\cdot)) \\
%    &= -2 -\beta(\barF(\Ginv(1-(\cdot)^{-1}))) + 0 
%    \leq -2 + \frac{\beta(\barF(x_R-(\cdot)^{-1}))}{\beta(\barF(x_R-(\cdot)^{-1}))-1}
%    < 0 %\numberthis \label{eq:alphaIfinite}
%\end{align*}
%implies that $\mathrm{I}(\r) = o(1)$ as $\rhotoone$.

Next, fix $0\leq \e<2-\frac{\beta(\barF)}{\beta(\barF)+1}$. Substitution of $w=\frac{\r}{1-\r}\frac{\nu}{1-\nu}$ in $\mathrm{II}(\r)$ yields
\begin{align*}
  \mathrm{II}(\r) 
    %&=\frac{2(1-\r)}{\barF(\Ginv(\r))} \int_{1-\r}^1 \Ginv\left(1-\frac{1-\r}{\r}\frac{1-\nu}{\nu}\right) \barF\left(\Ginv\left(1-\frac{1-\r}{\r}\frac{1-\nu}{\nu}\right)\right) \dd \nu \\
    &= \frac{2(1-\r)}{\barF(\Ginv(\r))} \int_1^\infty \frac{\r}{1-\r}\left(\frac{\r}{1-\r}+w\right)^{-2}\Ginv(1-w^{-1}) \barF(\Ginv(1-w^{-1})) \dd w \\
    &\leq \frac{2(1-\r)^{2-\e}}{\r^{1-\e}\barF(\Ginv(\r))} \int_1^\infty w^{-\e}\Ginv(1-w^{-1}) \barF(\Ginv(1-w^{-1})) \dd w.
\end{align*}
Let $q(w)$ denote the integrand in the last line. A similar analysis to \eqref{eq:alphaIinfinite} indicates that the term in front of the integral vanishes as $\rhotoone$, so we only need to show that the integral is bounded. This is implied by Lemma~\ref{lem:onesidedMatuszewskaPD}(i) after noting that
\begin{align*}
  \a(q)
    &\leq -\e + \a(\Ginv(1-(\cdot)^{-1})) + \a(\barF(\Ginv(1-(\cdot)^{-1}))) 
%    &\leq -\e-\frac{1}{\a(\barF)+1} - \frac{\a(\barF)}{\a(\barF)+1}
    \leq -1-\e
    < 0,
\end{align*}
where the inequalities follow from Lemmas~\ref{lem:onesidedMatuszewskaproduct} and \ref{lem:inverseMatuszewska2} and Corollary~\ref{cor:FGinvMatuszewska}.

Lastly, we wish to show that $\mathrm{III}(\r)=\Theta(1)$. Observe that
\begin{align*}
  \mathrm{III}(\r) \hspace{-2pt}&\hspace{2pt}
    \leq \l\E[B^2] \int_{1-\r}^{\frac{1}{1+\r}} \nu \, \frac{\barF\left(\Ginv\left(1-\frac{1-\r}{\r}\frac{1-\nu}{\nu}\right)\right)}{\barF(\Ginv(\r))} \dd \nu + \l\E[B^2] \int_{\frac{1}{1+\r}}^1 \frac{\barF\left(\Ginv\left(1-\frac{1-\r}{\r}\frac{1-\nu}{\nu}\right)\right)}{\barF(\Ginv(\r))} \dd \nu \\
    &\leq 2\r\E[B^*] \int_1^{\frac{1}{1-\r}} \frac{\r w}{1-\r} \left(\frac{\r}{1-\r}+w\right)^{-3} \frac{\barF(\Ginv(1-w^{-1}))}{\barF(\Ginv(\r))} \dd w + \E[B^*] \\
    &\leq \frac{2\E[B^*]}{\r} \int_1^{\frac{1}{1-\r}} \frac{w\barF(\Ginv(1-w^{-1}))}{\frac{1}{(1-\r)^2}\barF(\Ginv(\r))} \dd w + \E[B^*] 
    = \frac{2\E[B^*]}{\r} \int_1^{\frac{1}{1-\r}} \frac{f(w)/w}{f(1/(1-\r))} \dd w + \E[B^*],
\end{align*}
where $f(w) = w^2\barF(\Ginv(1-w^{-1}))$. Lemma~\ref{lem:onesidedMatuszewskaproduct} and Corollary~\ref{cor:FGinvMatuszewska} then state that $\beta(f) \geq 2-\frac{\beta(\barF)}{\beta(\barF)+1} %= \frac{\beta(\barF)+2}{\beta(\barF)+1}
> 0$% for $-\infty<\beta(\barF)\leq \a(\barF) < -2$
, and therefore Lemma~\ref{lem:onesidedMatuszewskaPI} implies 
\[
  \limsup_{\rhotoone} \int_1^{\frac{1}{1-\r}} \frac{f(w)/w}{f(1/(1-\r))} \dd v = \left[\liminf_{y\rightarrow \infty} \frac{f(y)}{\int_1^y f(w)/w \dd w} \right]^{-1} < \infty.
\]
As such, $\limsup_{\rhotoone} \mathrm{III}(\r)<\infty$.

In order to show $\liminf_{\rhotoone} \mathrm{III}(\r) >0$, fix $c\in(0,1)$ and let $\d_\r := (1-\r)/(c\r + 1-\r)$. One may then readily verify that $\mathrm{III}(\r) \geq \l m_2(\Ginv(1-c)) \int_{\d_\r}^{\frac{1}{1+\r}} \nu \dd \nu \rightarrow \frac{m_2(\Ginv(1-c))}{8\E[B]}>0$.

The $x_R=\infty$ case is concluded once we prove $\lim_{\rhotoone} h^*(\Ginv(\r)) = 0$. To this end, write $h^*(\Ginv(\r))$ as $x \barF(\Ginv(1-x^{-1}))/\E[B]$, where $x=(1-\r)^{-1}$. The claim then follows from Lemma~\ref{lem:onesidedMatuszewskavanish} after noting that
\begin{align*}
  \a(h^*(\Ginv(1-(\cdot)^{-1}))) \leq \a(\cdot) + \a(\barF(\Ginv(1-(\cdot)^{-1}))) \leq 1 - \frac{\a(\barF)}{\a(\barF) + 1} = \frac{1}{\a(\barF)+1} < 0,
\end{align*}
where the inequalities follow from Lemma~\ref{lem:onesidedMatuszewskaproduct} and Corollary~\ref{cor:FGinvMatuszewska}.

The $x_R<\infty$ case can be proven similarly. In that case, one fixes $1<\e<2-\frac{\beta(\barF(x_R-(\cdot)^{-1}))}{\beta(\barF(x_R-(\cdot)^{-1}))-1}$ and obtains $\a(\mathrm{I}(1-(\cdot)^{-1}) \leq -2 + \frac{\beta(\barF(x_R-(\cdot)^{-1}))}{\beta(\barF(x_R-(\cdot)^{-1}))-1}<0$, $\a(q) \leq -\e-\frac{\a(\barF(x_R-(\cdot)^{-1}))+1}{\a(\barF(x_R-(\cdot)^{-1}))-1} 
\leq 1-\e< 0$ and $\beta(f) \geq 2-\frac{\beta(\barF(x_R-(\cdot)^{-1}))}{\beta(\barF(x_R-(\cdot)^{-1}))-1}> 0$. The claim $h^*(\Ginv(\r)) \rightarrow \infty$ follows from Lemma~\ref{lem:xfx}.

%\alert{Entire analysis holds only if $\beta(\barF(x))>-\infty$.} \nts{However, if $\barF\in \MDA(\Lambda)$ then Proposition~0.12 in \citet{resnick1987extreme} claims that $\barF(\Ginv(1-x^{-1})) \in \RV_{-1}$. This will yield the (less-general) results.}

\subsubsection{Infinite variance} \label{subsubsec:infinitevariance}
Assume $\beta(\barF) > -2$ %Now, the intuition behind $\E[T_\fb] = \Theta\left(\frac{\E[B^*]\barF(\Ginv(\r))}{(1-\r)^2}\right)$ was that all jobs of size larger than $\Ginv(\r)$ would experience a sojourn time similar to the expected duration of a remaining busy period, which is of order $\E[(B\wedge x)^2]/(1-\r_x)^2 \approx \E[B^2]/(1-\r)^2$. Unfortunately, if $\b(\barF(x)) > -2$ than one may show that $\E[B^2]=\infty$ and as such the expected duration of a remaining busy period is of order $\omega(1/(1-\r)^2)$. At the same time, one may show that $\barF(\Ginv(\r))=o((1-\r)^2)$. It is therefore not immediately clear how $\E[T_\fb]$ may be expected to behave.
and recall that $m_2(x) = 2\E[B] \int_0^x t \dd G(t) = 2\E[B] \left( \int_0^x \barG(t)\dd t - x\barG(x) \right)$. By Lemmas~\ref{lem:onesidedMatuszewskaproduct} and \ref{lem:barGMatuszewska}, one sees that $\beta((\cdot)\barG(\cdot))>0$ and therefore it follows from Lemma~\ref{lem:onesidedMatuszewskaPI} that
\begin{equation}
  \limsup_{\xtoinfty} \frac{m_2(x)}{2\E[B]x\barG(x)} = \limsup_{\xtoinfty} \frac{\int_0^x \barG(t)\dd t}{x\barG(x)} - 1 < \infty.
  \label{eq:m2x}
\end{equation}
Also, since $\beta((\cdot)\barF(\cdot))>-\infty$, Lemma~\ref{lem:onesidedMatuszewskaPD}(ii) indicates that 
\[ 
  \limsup_{\xtoinfty} \frac{x\barF(x)}{\barG(x)} = \limsup_{\xtoinfty} \frac{\E[B]x\barF(x)}{\int_x^\infty \barF(t) \dd t} < \infty.
\]
Consequently, it follows from relation \eqref{eq:Tmain} that, for some $C,D>0$ and all $\rho$ sufficiently close to one, we have
\begin{align*}
  \E[T_\fb]
    &\leq \frac{\E[B]}{\r} \log\frac{1}{1-\r} + 2 \int_0^\infty \frac{x \barF(x)}{(1-\r G(x))^2}\dd G(x) + \frac{1}{\E[B]} \int_0^\infty \frac{m_2(x)}{x\barG(x)}\frac{x \barF(x)}{(1-\r G(x))^2} \dd G(x) \\
    &\leq \frac{\E[B]}{\r} \log\frac{1}{1-\r} + C \int_0^\infty \frac{x \barF(x)}{\barG(x)} \frac{1}{1-\r G(x)} \dd G(x) 
    \leq D \log\frac{1}{1-\r},
\end{align*}
and therefore $\E[T_\fb] = \Theta\left(\log\frac{1}{1-\r}\right)$.

%As to gain understanding of this result, assume that $\barF(x) = x^{-\a}, x>1, \a\in(1,2)$. Then we roughly expect all jobs of size $x\geq \Ginv(\r)$ to experience a `waiting time' of length $\E[(B\wedge x)^2]/(1-\r_x)^2 \approx \Theta(m_2(\barG(\r))/(1-\r)^2) = \Theta( \Ginv(\r) \barG(\Ginv(\r))/(1-\r)^2 ) = \Theta((1-\r)^{-\a/(\a-1)})$. The fraction of jobs that experience this waiting time is given by $\barF(\Ginv(\r)) =\Theta((1-\r)^{\a/(\a-1)})$, so that the product of these terms does not grow polynomial in $(1-\r)$. The calculations above specify the growth of the waiting time to be of order $\log(1/(1-\r))$, which agrees with the order of growth of the `service time' of a job that is at least a certain threshold (where the threshold grows polynomially in $(1-\r)$).

\subsection{Special cases}
This section proves Theorem~\ref{thm:ETextreme}. The maximum domain of attraction of each of the extreme value distributions are considered in order, followed by a distribution with an atom in its right endpoint. The Fr{\'e}chet and Weibull cases follow rapidly from Theorem~\ref{thm:ETMatuszewska} and the Dominated Convergence Theorem. The same approach works for the Gumbel case, although Theorem~\ref{thm:ETMatuszewska} is not directly applicable. Finally, the atom case follows readily by analysing the sojourn time of maximum-sized jobs.

\subsubsection[Fr\'echet(a) and Weibull(a)]{Fr{\'e}chet($\a$) and Weibull($\a$)} \label{subsubsec:ETFrechet}
Theorem~3.3.7 in \citet{embrechts1997modelling} states that $F\in\MDA(\Phi_\a)$ if and only if $\barF(x)=L(x)x^{-\a}$ is regularly varying with index $-\a$. Karamata's theorem \citep[Theorem~1.5.11]{bingham1989regular} then states that $\E[B]\bar{G}(x)\sim x\barF(x)/(\a-1)$ %\sim L(x) x^{1-\a}/(\a-1)
is regularly varying with index $-(\a-1)$. %, from which it follows that $\lim_{\rhotoone} h^*(\Ginv(\r)) = \lim_{\rhotoone} \frac{\a-1}{\Ginv(\r)} = 0$. 
% and that $m_2(x) = 2\int_0^x y\barF(y)\dd y \sim 2 x^2 \barF(x)/(2+\a)$ as $x\rightarrow\infty$. %It follows that $h^*(x) \sim \frac{\a-1}{x}$.
Consequently, Theorem~1.5.12 in \citet{bingham1989regular} states that $\Ginv(1-1/x)$ is regularly varying with index $1/(\a-1)$ and therefore $\barF(\Ginv(1-1/x))$ is regularly varying with index $-\a/(\a-1)$ \citep[Proposition~1.5.7]{bingham1989regular}.
%$\inv{\barG}(y) \sim y^{-1/(\a-1)}$ as $y\downarrow 0$, hence $\inv{G}(1-1/x) = \inf\{y\geq 0: G(y)> 1-1/x\} = \inf\{y\geq 0: \barG(y)< 1/x\} = \inv{\barG}(1/x) \sim x^{1/(\a-1)}$ as $\xtoinfty$.

First assume $\a>2$. We saw in Section~\ref{subsubsec:finitevariance} that the asymptotic behaviour of $\E[T_\fb]$ is identical to the asymptotic behaviour of term $\mathrm{III}(\r)$ (cf.\ relation~\eqref{eq:meanMatuszewskadecomposed}). Now, the Uniform Convergence Theorem \citep[Theorem~1.5.2]{bingham1989regular} states that $\frac{\barF(\Ginv(1-1/x))}{\barF(\Ginv(1-1/y))} \rightarrow \left(\frac{y}{x}\right)^{\a/(\a-1)}$ uniformly for all $0<c<x,y<\infty$. Therefore, we substitute $w=\frac{\nu-(1-\r)}{\r}$ and exploit the Dominated Convergence Theorem to obtain
\begin{align*}
  \lim_{\rhotoone} \mathrm{III}(\r) \hspace{-26pt} & \hspace{26pt} \\
    &= \lim_{\rhotoone} \frac{\r^2}{\E[B]} \int_0^1 (\r w + 1-\r) \, m_2\left(\Ginv\left(1-\frac{(1-\r)(1-w)}{1-\r+\r w}\right)\right) \frac{\barF\left(\Ginv\left(1-\frac{(1-\r)(1-w)}{1-\r+\r w}\right)\right)}{\barF(\Ginv(\r))} \dd w \\
    &= \frac{\E[B^2]}{\E[B]} \int_0^1 w \left(\frac{1-w}{w}\right)^{\a/(\a-1)} \dd w 
    = \E[B^*]\frac{\pi/(\a-1)}{\sin(\pi/(\a-1))} \frac{\a}{\a-1}.
\end{align*}

Similarly, Theorem~3.3.12 in \citet{embrechts1997modelling} states that $F\in\MDA(\Psi_\a),\a>0,$ if and only if $x_R<\infty$ and $\barF(x_R-x^{-1})=L(x)x^{-\a}$ is regularly varying with index $-\a$. The corresponding result then follows after noting that $\E[B]\barG(x_R-x^{-1})\sim L(x)x^{-\a-1}/(\a+1)$ is regularly varying with index $-(\a+1)$ and $\frac{\barF(\Ginv(1-1/x))}{\barF(\Ginv(1-1/y))} \rightarrow \left(\frac{y}{x}\right)^{\a/(\a+1)}$ uniformly for all $0<c<x,y<\infty$.

Finally, assume $F\in\MDA(\Phi_\a),\a\in(1,2)$. Then, Karamata's Theorem implies $m_2(x) = 2\int_0^x y\barF(y)\dd y \sim 2 x^2 \barF(x)/(2-\a)$ as $x\rightarrow\infty$. We analyse relation~\ref{eq:Tmain} and again exploit the Dominated Convergence Theorem to find
\begin{align*}
  \E[T_\fb] \hspace{-16pt}&\hspace{16pt} 
    %&= \frac{1}{\l} \log\frac{1}{1-\r} + 2\r \int_0^\infty \frac{x \barF(x)}{(1-\r_x)^2}\dd G(x) + \frac{\r^2}{\E[B]} \int_0^\infty \frac{m_2(x) \barF(x)}{(1-\r_x)^3} \dd G(x) \\
    = \frac{\E[B]}{\r} \log\frac{1}{1-\r} 
      + 2\r \int_0^\infty \frac{x \barF(x)}{\barG(x)} \frac{1-G(x)}{(1-\r G(x))^2}\dd G(x) \\
        &\hspace{16pt} \qquad + \frac{\r^2}{\E[B]} \int_0^\infty \frac{m_2(x) \barF(x)}{\barG(x)^2} \frac{(1-G(x))^2}{(1-\r G(x))^3} \dd G(x) \\
    &\sim \E[B] \log\frac{1}{1-\r} 
      + 2(\a-1) \E[B] \int_0^1 \frac{1-y}{(1-\r y)^2}\dd y 
        + \frac{2}{\E[B]} \frac{(\a-1)^2\E[B]^2}{2-\a} \int_0^1 \frac{(1-y)^2}{(1-\r y)^3}\dd y \\
    &\sim \E[B] \log\frac{1}{1-\r} 
      + 2(\a-1)\E[B] \log\frac{1}{1-\r} 
        + \frac{2(\a-1)^2\E[B]}{2-\a} \log\frac{1}{1-\r} 
    = \frac{\a}{2+\a} \E[B] \log\frac{1}{1-\r}
\end{align*}
as $\rhotoone$.

\subsubsection{Gumbel} \label{subsubsec:ETGumbel}
%Definitions of $\Pi$- and $\Gamma$-varying functions can be found in Section~\ref{sec:TmeanGumbel}
%
%According to Propositions~1.9 in \citet{resnick1987extreme} $F\in\MDA(\Lambda)$ if and only if $U_F(\cdot) :=1/\barF(\cdot)\in\Gamma$ with auxiliary function $f_F(\cdot) = 1/h^*(\cdot)$. Proposition~0.9(a) then states that $V_F(y):= \Uinv_F(y) = \inv{\left(\frac{1}{\barF}\right)}(y) = \Finv\left(1-\frac{1}{y}\right)\in\Pi$ with auxiliary function $a_F(x)=f_F(\Uinv_F(x)) = 1/h^*(\Finv(1-1/x))$. 
%
%Since $F\in\MDA(\Lambda)$ implies $G\in\MDA(\Lambda)$ with auxiliary function $h^*$, we find that $U_G(\cdot):=1/\barG(\cdot)\in\Gamma$ with auxiliary function $f_G(\cdot) = 1/h^*(\cdot)$, and $V(y):=\Uinv_G(y) = \Ginv\left(1-\frac{1}{y}\right)\in\Pi$ with auxiliary function $a_G(x) = 1/h^*(\Ginv(1-1/x))$. 
%
%Proposition~0.12 states that any auxiliary function $a(\cdot)$ of a function $V\in\Pi$ is $0$-varying. In particular, $a_G(x) = \frac{1}{h^*(\Ginv(1-1/x))} = \frac{\E[B]}{x\barF(\Ginv(1-1/x))}$ is $0$-varying and as a consequence, $\barF(\Ginv(1-1/x))$ is $(-1)$-varying.
%

If $F\in\MDA(\Lambda)$, then so is $G$ by Lemma~\ref{lem:FandGareGumbel} and we may choose $h^*$ as the auxiliary function of $G$. Propositions~0.9(a), 0.10 and 0.12 in \citet{resnick1987extreme} together state that 
\[
  a_G(x):=\frac{1}{h^*(\Ginv(1-1/x))} = \frac{\E[B]}{x\barF(\Ginv(1-1/x))}
\] 
is $0$-varying\footnote{The propositions regard $\Pi$- and $\Gamma$-varying functions; we consider these classes in Section~\ref{sec:TmeanGumbel}.}, implying that $\barF(\Ginv(1-1/x))$ is $(-1)$-varying.

Following the analysis in Section~\ref{subsubsec:finitevariance}, we obtain $\a(\mathrm{I})=-1<0$ as before. Consider term $\mathrm{II}(\r)$. By Markov's inequality, we have $\barG(x)\leq \E[B^*]/x$. Substituting $x = \Ginv(1-w^{-1})$ then yields $\Ginv(1-w^{-1})\leq \E[B^*]w$, and hence
\begin{align*}
  \mathrm{II}(\r) 
    &= \frac{2(1-\r)}{\barF(\Ginv(\r))} \int_1^\infty \frac{\r}{1-\r}\left(\frac{\r}{1-\r}+w\right)^{-2}\Ginv(1-w^{-1}) \barF(\Ginv(1-w^{-1})) \dd w \\
    &\leq \frac{2\E[B^*](1-\r)^{3/2}}{\r^{1/2}\barF(\Ginv(\r))} \int_1^\infty w^{1/2} \barF(\Ginv(1-w^{-1})) \dd w.
\end{align*}
The term in front of the integral and the integrand both have upper Matuszewska index $-1/2$, and therefore $\mathrm{II}(\r)\rightarrow 0$.

Lastly, consider term $\mathrm{III}(\r)$. The relation $\limsup_{\rhotoone} \mathrm{III}(\r) < \infty$ follows analogously to the analysis in Section~\ref{subsubsec:finitevariance}. Then, along the lines of Section~\ref{subsubsec:ETFrechet}, one may apply the Uniform Convergence Theorem and the Dominated Convergence Theorem to derive the theorem statement.
%the Uniform Convergence Theorem \citep[Theorem~1.5.2]{bingham1989regular} states that $\frac{\barF(\Ginv(1-1/x))}{\barF(\Ginv(1-1/y))} \rightarrow \frac{y}{x}$ uniformly over all $0<c<x,y<\infty$. Therefore, we substitute $w=\frac{\nu-(1-\r)}{\r}$ and exploit the Dominated Convergence Theorem to obtain
%\begin{align*}
%  \lim_{\rhotoone} \mathrm{III}
%    &= \lim_{\rhotoone} \frac{\r^2}{\E[B]} \int_0^1 (\r w + 1-\r) \, m_2\left(\Ginv\left(1-\frac{(1-\r)(1-w)}{1-\r+\r w}\right)\right) \frac{\barF\left(\Ginv\left(1-\frac{(1-\r)(1-w)}{1-\r+\r w}\right)\right)}{\barF(\Ginv(\r))} \dd w \\
%    &= \frac{\E[B^2]}{\E[B]} \int_0^1 (1-w) \dd w
%    = \E[B^*].
%\end{align*}

\subsubsection{Atom in right endpoint}
First, we show that $\mathrm{I}(\r) + \mathrm{II}(\r) = o(1)$. Lemma~\ref{lem:xfx} states that $\lim_{x \uparrow x_R} h^*(x) = \infty$, and therefore 
%by Lemma
%$
% \E[B]\barG(x) 
%   = \int_x^{x_R} \barF(t) \dd t 
%   \leq (x_R-x)\barF(x),
%$
%implying that $h^*(x) \rightarrow \infty$ as $x\uparrow x_R$. 
$\lim_{\rhotoone} \mathrm{I}(\r) = \lim_{\rhotoone} \frac{(1-\r) \log\frac{1}{1-\r}}{\r h^*(\Ginv(\r))} = 0$. Also, $\Ginv$ is bounded from above by $x_R$ and consequently 
$\lim_{\rhotoone} \mathrm{II}(\r) 
  \leq \lim_{\rhotoone} \frac{2(1-\r)}{\barF(\Ginv(\r))} \cdot x_R
  = \lim_{\rhotoone} \frac{2x_R}{\E[B]h^*(\Ginv(\r))} 
  = 0
$. 

It remains to show that $\mathrm{III}(\r) \rightarrow \E[B^*]$ and $\barF(\Ginv(\r)) \rightarrow p$ as $\rhotoone$. The following lemma facilitates the analysis of this term. The proof the lemma is postponed until the end of this section.

\begin{lemma} \label{lem:linearapprox}
%Let $f: D \rightarrow \R$ be any function that maps $D\subseteq \R$ onto $\R$. Let $(x_n)_{n\in\N}$ be a sequence in $D$ such that $x_n\uparrow x\in\R$ and assume that $\lim_{n\rightarrow \infty} f(x_n) = p$. Then, there exist $q>0$ and $z>0$ such that
Let $f: D \rightarrow \R$ be any function that maps $D\subseteq \R$ onto $\R$, and assume that $\lim_{y\uparrow x} f(y) = p$ for some $x$ in the closure $\bar{D}$ of $D$. Then, there exist $z>0$ and $q>0$ such that
\begin{equation}
  f(x-y) \leq p + q y
  \label{eq:linearapprox}
\end{equation}
for all $y\in(0,z]$ that satisfy $x-y\in D$.
\end{lemma}

Let $q>0$ and $\d^*>0$ be such that $\barF(x_R-\d) \leq p + q\d$ for all $\d\in(0,\d^*]$. It follows that $\E[B]\barG(x) = \int_x^{x_R} \barF(y) \dd y \sim p(x_R-x)$ as $x\uparrow x_R$, and hence $x_R-\Ginv(u) \sim \E[B](1-u)/p$ as $u\uparrow 1$. Fix $\e>0$ and let $u^*\in(0,1)$ be such that $x_R - \Ginv(u)\leq (1+\e)\E[B](1-u)/p$ for all $u\in(u^*,1)$. Now, for all $u>\r_0:=\max\{u^*, 1-p\d^*/((1+\e)\E[B])\}$ we have
\begin{equation}
  p \leq \barF(\Ginv(u)) \leq p + \frac{q}{p}(1+\e)\E[B](1-u) =: p + p\tilde{q}(1-u)
  \label{eq:FbarGinvAtom}
\end{equation}
and hence, for $\tilde{q}=q(1+\e)\E[B]/p^2$, the relations
%\begin{equation*}
%  \frac{1}{1+\tilde{q}(1-v)} \leq \frac{\barF(\Ginv(u))}{\barF(\Ginv(v))} \leq 1 + \tilde{q}(1-u)
%\end{equation*}
%for all $u,v>\r_0$. It follows that
\begin{equation*}
  \frac{1}{1+\tilde{q}(1-\r)} 
    \leq \frac{\barF\left(\Ginv\left(1-\frac{1-\r}{\r}\frac{1-\nu}{\nu}\right)\right)}{\barF(\Ginv(\r))} 
      \leq 1 + \tilde{q}\,\frac{1-\r}{\r}\frac{1-\nu}{\nu} 
        \leq 1 + \tilde{q}\,\frac{1-\r}{\r}\frac{1}{\nu}
\end{equation*}
hold for all $\nu>\frac{1-\r}{1-\r \cdot \r_0}, \r>\r_0$.

Consider term $\mathrm{III}(\r)$. On the one hand, we find
\begin{align*}
  \limsup_{\rhotoone} \mathrm{III}(\r)
    &\leq \limsup_{\rhotoone} \frac{\E[B^2]}{\E[B]} \int_{1-\r}^1 \nu \, \frac{\barF\left(\Ginv\left(1-\frac{1-\r}{\r}\frac{1-\nu}{\nu}\right)\right)}{\barF(\Ginv(\r))} \dd \nu \\
    &\leq \limsup_{\rhotoone} \frac{\E[B^2]}{\E[B]} \int_{1-\r}^{\frac{1-\r}{1-\r\cdot \r_0}} \frac{1}{p} \dd \nu + \frac{\E[B^2]}{\E[B]} \int_{\frac{1-\r}{1-\r\cdot\r_0}}^1 \left\{\nu + \tilde{q}\,\frac{1-\r}{\r}\right\} \dd \nu \\
    &\leq \limsup_{\rhotoone} \frac{\E[B^2]}{p\E[B]} \frac{1-\r}{1-\r\cdot \r_0} + \frac{\E[B^2]}{2\E[B]} + \frac{\E[B^2]}{\E[B]} \,\tilde{q}\, \frac{1-\r}{\r}
    = \E[B^*].
\end{align*}
On the other hand, we have
\begin{align*}
  \liminf_{\rhotoone} \mathrm{III}(\r)
    &\geq \liminf_{\rhotoone} \frac{\r m_2(\Ginv(\r_0))}{\E[B]} \int_{\frac{1-\r}{1-\r\cdot\r_0}}^1 \frac{\nu}{1+\tilde{q}(1-\r)} \dd \nu \\
    &= \liminf_{\rhotoone} \frac{\r m_2(\Ginv(\r_0))}{2\E[B]} \frac{1}{1+\tilde{q}(1-\r)} \left(1-\left(\frac{1-\r}{1-\r\cdot\r_0}\right)^2\right)
    = \frac{m_2(\Ginv(\r_0))}{\E[B^2]}\cdot \E[B^*].
\end{align*}
Since $\r_0$ may be chosen arbitrarily close to unity, we find $\E[T_\fb] \sim \frac{\E[B^*]\barF(\Ginv(\r))}{(1-\r)^2} \sim \frac{p\E[B^*]}{(1-\r)^2}$ as $\rhotoone$. Here, the last equivalence follows from \eqref{eq:FbarGinvAtom}.
%
%Note that in the case $p=1$ the second moment of $B$ equals $\E[B]^2$ and we retrieve the known result $\E[T_\fb] \sim \frac{\E[B]}{2(1-\r)^2}$. \textit{This agrees with the upper bound $\E[T_\fb] \leq \frac{\E[B]}{2}\frac{2-\r}{(1-\r)^2}$ from Theorem~5.1 in \citet{nuyens2008foreground}.}
%
%
The section is concluded with the proof of Lemma~\ref{lem:linearapprox}.
\begin{proof}[Proof of Lemma~\ref{lem:linearapprox}]
Without loss of generality, we assume that $(x-1,x)\subset D$. For sake of finding a contradiction, assume that the above statement is not true, i.e.\ for all $z>0$ and all $q>0$ there exists $\xi\in(0,z]$ such that
\begin{equation}
  f(x-\xi) > p + q \xi.
  \label{eq:flarger}
\end{equation}
Define $z_1 := 1, q_1 := 1$ and let $\xi_1\in(0,1]$ be such that~\eqref{eq:flarger} holds with $q=q_1$ and $\xi=\xi_1$. By definition of the left limit, for any $\e>0$ there exists $\eta^*>0$ such that $f(x-\eta) \leq p+\e$ for all $\eta\in(0,\eta^*]$. In particular, by choosing $\e = q_1 \xi_1$ we obtain $\eta^*=:\eta_2^*<\xi_1\leq z_1$ such that $f(x-\eta) \leq p+q_1 \xi_1$ for all $\eta \in(0,\eta_2^*]$.

Define $z_2 := \min\{\eta_2^*, 1/2\}$ and set $q_2:= 1/z_2$. Again, there exists $\xi_2\in(0,z_2]$ such that~\eqref{eq:flarger} holds for $q=q_2$ and $\xi=\xi_2$. By repeating the above procedure we obtain three sequences $(q_n)_{n\in\N}, (z_n)_{n\in\N}$ and $(\xi_n)_{n\in\N}$ such that $q_n = 1/z_n$, $0< z_{n+1} < \xi_n < z_n \leq 1/n$ and
\begin{equation}
  f(x-\xi_n) > p + q_n \xi_n
  \label{eq:flargergeneral}
\end{equation}
for all $n\in\N$. From these properties, one may additionally deduce that $\xi_n > 1/q_{n+1}, \xi_n\downarrow 0$ and $q_n\rightarrow\infty$. 

We will obtain a contradiction by showing that $(q_n)_{n\in\N}$ also converges as $n\rightarrow\infty$. Assume that $\limsup_{n\rightarrow\infty} q_n\xi_n > 0$. Then, relation~\eqref{eq:flargergeneral} implies $\limsup_{n\rightarrow\infty} f(x-\xi_n) \geq  \limsup_{n\rightarrow\infty} p + q_n \xi_n > p$ which contradicts the lemma assumptions. Therefore, $\limsup_{n\rightarrow\infty} q_n\xi_n$ must equal zero. It then follows that $0\leq \limsup_{n\rightarrow\infty} q_n/q_{n+1} \leq \limsup_{n\rightarrow\infty} q_n\xi_n = 0$ and as such the ratio test states that the sequence $(q_n)_{n\in\N}$ converges. %However, this contradicts with our earlier observation that $q_n$ diverges to infinity.
\end{proof}
Note that Lemma~\ref{lem:linearapprox} can be applied generally to yield lower and upper bounds for $f(y)$ around any point $x\in \bar{D}$ for which either $\lim_{y\uparrow x} f(y)$ or $\lim_{y\downarrow x} f(y)$ exists.

%Don't do Frechet(2).
%If L(x)=1 (pareto), then IIb \approx \log^2 \frac{1}{1-\r}
%If L(x)=log(x), then IIb \approx \log^3 \frac{1}{1-\r}
%If L(x)=1/log(x), then IIb \approx \log \frac{1}{1-\r} \log\log \frac{1}{1-\r}

\section[Asymptotic behaviour of h*(Ginv(rho)) if F in MDA(Lambda)]{Asymptotic behaviour of $h^*(\Ginv(\r))$ if $F\in\MDA(\Lambda)$}
\label{sec:TmeanGumbel}
This section is dedicated to the proof of Theorem~\ref{thm:TmeanGumbel}. Theorem~\ref{thm:MDAGumbel} states that $c_n$ may be chosen as $1/h^*(\Finv(1-n^{-1}))$, so that the theorem follows from Theorem~\ref{thm:ETextreme} after analysing the limit $\lim_{n\rightarrow \infty} h^*(\Ginv(1-n^{-1}))/h^*(\Finv(1-n^{-1})) = \lim_{y\uparrow 1} (1-y)^{-2} \barF(\Ginv(y))\barG(\Finv(y))$. The proof heavily relies upon the work by \citet{dehaan1974equivalence} and \citet{resnick1987extreme}, who both consider $\Gamma$- and $\Pi$-varying functions:
\begin{definition} \label{def:Gammavarying}
A function $U:(x_L,x_R)\rightarrow \R, \lim_{x\uparrow x_R} U(x) = \infty$ is in the class of $\Gamma$-varying functions if it is non-decreasing, and there exists a function $f:(x_L,x_R)\rightarrow\R_{\geq 0}$ satisfying
\begin{equation}
  \lim_{x\uparrow x_R} \frac{U(x+tf(x))}{U(x)} = e^t
\end{equation}
for all $t\in\R$. The function $f(\cdot)$ is called an \textit{auxiliary function} and is unique up to asymptotic equivalence.
\end{definition}
\begin{definition} \label{def:Pivarying}
A function $V:(x_L,\infty)\rightarrow \R_{\geq 0}$ is in the class of $\Pi$-varying functions if it is non-decreasing, and there exist functions $a(x)>0, b(x)\in\R$, such that
\begin{equation}
  \lim_{\xtoinfty} \frac{V(tx)-b(x)}{a(x)} = \log t
\end{equation}
for all $t\in\R$. The function $a(\cdot)$ is called an \textit{auxiliary function} and is unique up to asymptotic equivalence.
\end{definition}

It turns out that $\Gamma$- and $\Pi$-varying functions are closely related to $\MDA(\Lambda)$. In particular, if $F\in\MDA(\Lambda)$ with auxiliary function $1/h^*$, then Proposition~1.9 in \citet{resnick1987extreme} states that $U_F :=1/\barF\in\Gamma$ with auxiliary function $f_F := 1/h^*$. Proposition~0.9(a) then states that $V_F(\cdot):= \Uinv_F(\cdot) = \inv{\left(\frac{1}{\barF}\right)}(\cdot) = \Finv\left(1-(\cdot)^{-1}\right)\in\Pi$ with auxiliary function $a_F(\cdot):=f_F(\Uinv_F(\cdot)) = 1/h^*(\Finv(1-(\cdot)^{-1}))$. Similarly, using Lemma~\ref{lem:FandGareGumbel}, we find that $U_G:=1/\barG\in\Gamma$ and $V_G(\cdot):=\Uinv_G(\cdot) = \Ginv\left(1-(\cdot)^{-1}\right)\in\Pi$ with auxiliary function $a_G(\cdot) := 1/h^*(\Ginv(1-(\cdot)^{-1}))$.

Now, since Theorem~\ref{thm:MDAGumbel} states that the norming constants $c_n$ may be chosen as $1/h^*(\Finv(1-1/n))$, we are done once we show that $\lim_{n \rightarrow \infty} c_n h^*(\Ginv(1-1/n)) = \lim_{\xtoinfty} \frac{a_F(x)}{a_G(x)}$ tends to the right quantity for all cases in the theorem.

Corollary~3.4 in \citet{dehaan1974equivalence} states that\footnote{Here, we denote $0^{-1}=+\infty$.} $\lim_{x \uparrow x_R} \frac{a_F(x)}{a_G(x)} = \xi^{-1} \in [0,\infty]$ if and only if there exist a positive function $b(x)$ with $\lim_{x \uparrow x_R}b(x) = \xi$ and constants $b_2>0$ and $b_3\in\R$ such that\footnote{Their paper only considers the $x_R=\infty$ case; however, the proof also holds for finite $x_R$.} $P(x) = b_3 + \int_0^x b(t) \dd t$ and $\inv{V_F}(x) \sim b_2 \inv{V_G}(P(x))$ as $x \uparrow x_R$. As $\inv{V_\bullet}(x) = \inv{(\inv{U_\bullet})}(x) \sim U_\bullet(x)$ \citep[p.44]{resnick1987extreme}, this is equivalent to finding a function $P(x)$, of the given form, that satisfies
\begin{equation}
  \lim_{x\uparrow x_R} \frac{\barG(P(x))}{b_2 \barF(x)} = \lim_{x\uparrow x_R} \frac{U_F(x)}{b_2 U_G(P(x))} = \lim_{x\uparrow x_R} \frac{\inv{V}_F(x)}{b_2 \inv{V}_G(P(x))} = 1.
  \label{eq:P2}
\end{equation}
We use the following lemma, proven at the end of this section, to construct a suitable $P(x)$: 
\begin{lemma} \label{lem:strictlyincreasingF}
Let $F$ be a c.d.f. Then, there exists a strictly increasing, continuous c.d.f.\ $F_{\uparrow}(x)$ satisfying both $\barF_{\uparrow}(x)\sim \barF(x)$ and $\barG(F_{\uparrow}(x))\sim \barG(F(x))$ as $x\uparrow x_R$.
\end{lemma}

As $\Ginv(F_{\uparrow}(x))$ is strictly increasing, there exists a positive function $b(\cdot)$ such that $\int_0^x b(t) \dd t=\Ginv(F_{\uparrow}(x))$. Therefore, we see that \eqref{eq:P2} is satisfied with $b_2=1$ and $b_3=0$. The result follows once we show that
\begin{equation}
  \lim_{\xtoinfty} b(x) = \lim_{\xtoinfty} \frac{P(x)}{x} = \lim_{\xtoinfty} \frac{\Ginv(F(x))}{x} = \xi
  \label{eq:Pinfinite}
\end{equation}
if $x_R=\infty$, and once we show that
\begin{equation}
  \lim_{x \uparrow x_R} b(x) = \lim_{x\uparrow x_R} \frac{P(x_R)-P(x)}{x_R-x} = \lim_{x\uparrow x_R} \frac{x_R - \Ginv(F(x))}{x_R-x} = \xi.
  \label{eq:Pfinite}
\end{equation}
if $x_R<\infty$.

The right-hand sides of both \eqref{eq:Pinfinite} and \eqref{eq:Pfinite} depend on the function $\Ginv(F(x))$. The advantage of this representation is apparent from the following key relation, which connects $\Ginv(F(x))$ to $h^*(x)$:
\begin{equation}
  \E[B] h^*(x) = \exp\left[\int_{\Ginv(F(x))}^x h^*(t) \dd t\right].
  \label{eq:GinvF}
\end{equation}
Relation~\eqref{eq:GinvF} follows readily from $h^*(x) = -\frac{\dd}{\dd x} \log \barG(x)$. In the upcoming analysis, we first focus on \eqref{eq:Pinfinite} and then consider \eqref{eq:Pfinite}.

\subsection{Infinite support}
First assume $x_R=\infty$. The following theorem relates the assumptions on $\barF(x)$ to properties of $h^*(x)$:
\begin{theorem}[\citet{beirland1995}, Theorem~2.1] \label{thm:beirland}
\leavevmode
\begin{enumerate}[(i)]
\item If there exists $\a>0$ and a slowly varying function $l(x)$ such that $-\log \barF(x)\sim l(x) x^\a$ as $x\rightarrow\infty$, then $h^*(x)\sim \a l(x) x^{\a-1}$ as $x\rightarrow\infty$ if and only if 
  \begin{equation}
    \lim_{\l \downarrow 1} \liminf_{x\rightarrow\infty} \inf_{t\in[1,\l]} \{\log h^*(tx) - \log h^*(x)\} \geq 0.
  \end{equation}
\item If there exists a function $l(x):[0,\infty)\rightarrow \R, \liminf_{x\rightarrow \infty} l(x) >1$ such that for all $\l>0$
  \begin{equation}
    \lim_{x\rightarrow\infty} \frac{-\log\barF(\l x)+\log\barF(x)}{l(x)} = \log(\l),
  \end{equation}
  then $l(x)$ is slowly varying and $h^*(x)\sim (l(x)-1)/x$ as $x\rightarrow\infty$.
\end{enumerate}
\end{theorem}

The cases in Theorem~\ref{thm:TmeanGumbel} correspond to the cases in Theorem~\ref{thm:beirland}. We will consider the implications of Theorem~\ref{thm:beirland} as to derive the results presented in Theorem~\ref{thm:TmeanGumbel}.
\begin{enumerate}[(i)]
\item
Assume $h^*(x) \sim \a l(x) x^{\a-1},\a>0,$ and note that
\begin{align*}
  \lim_{x\rightarrow \infty} \frac{-\log(\E[B]h^*(x))}{x h^*(x)} 
    &= \lim_{x\rightarrow \infty} \frac{-\log(\E[B]\a l(x)) - (\a-1)\log(x)}{\a l(x)x^\a} 
    = 0.
\end{align*}
%and
%\begin{align*}
%  \lim_{x\rightarrow \infty} \frac{-\log(\E[B]h^*(x))}{x h^*(x)} 
%	&= \lim_{x\rightarrow \infty} \left[\frac{-\log(\E[B])-\log(k(x)\log(x))}{k(x)\log(x)} + \frac{\log(x)}{x h^*(x)}\right] \\
%    &= \lim_{n\rightarrow \infty} \frac{\log(\Finv(1-n^{-1}))}{\Finv(1-n^{-1}) h^*(\Finv(1-n^{-1}))}
%    = \lim_{n\rightarrow \infty} \frac{c_n\log(d_n)}{d_n},
%\end{align*}
%which proves the first assertion. 
%
We will prove the relation $\lim_{x\rightarrow \infty} \Ginv(F(x))/x = 1$ by contradiction. Specifically, if $\limsup_{x\rightarrow \infty} \Ginv(F(x))/x >1$ then there exists $\e>0$ and a sequence $(x_n)_{n\in\N}, x_n\rightarrow\infty$, such that $\Ginv(F(x_n))/x_n \geq 1+\e$ for all $n\in\N$. The Uniform Convergence Theorem \citep[Theorems~1.2.1 and 1.5.2]{bingham1989regular} then implies
\begin{align*}
  \frac{-\log(\E[B]h^*(x_n))}{x_n h^*(x_n)}
    &= \int_{x_n}^{\Ginv(F(x_n))} \frac{h^*(t)}{x_n h^*(x_n)} \dd t
    = \int_1^{\Ginv(F(x_n))/x_n} \frac{h^*(\tau x_n)}{h^*(x_n)} \dd \tau \\
    &\geq \int_1^{1+\e} \frac{h^*(\tau x_n)}{h^*(x_n)} \dd \tau
    \sim \int_1^{1+\e} \tau^{\a-1} \dd \tau 
    = \a^{-1} ((1+\e)^\a-1)
\end{align*}
for every $n\in\N$. However, this contradicts with $\lim_{x\rightarrow \infty} \frac{\log(\E[B]h^*(x))}{x h^*(x)} = 0$ and it follows that $\liminf_{x\rightarrow \infty} \Ginv(F(x))/x \leq 1$. Similarly, one may show that $\liminf_{x\rightarrow\infty} \Ginv(F(x))/x \geq 1$ and therefore $\lim_{x\rightarrow \infty} \Ginv(F(x))/x = 1$ as claimed.

\item
Alternatively, assume $h^*(x)\sim (l(x)-1)/x$ and denote $L=\lim_{x\rightarrow\infty} \log(x)/l(x) \in [0,\infty]$. Then Lemma~\ref{lem:xfx} states that $l(x)\rightarrow \infty$ and as such
\begin{align*}
  \lim_{x\rightarrow \infty} \frac{-\log(\E[B]h^*(x))}{x h^*(x)} 
    &= \lim_{x\rightarrow \infty} \frac{-\log(\E[B]) - \log(l(x)-1) + \log(x)}{l(x)-1} 
    = L. \numberthis \label{eq:ratiotoL}
\end{align*}
Now, if $L=0$ then the analysis in (i) yields $\lim_{x\rightarrow \infty} \Ginv(F(x))/x = 1 = e^0$. If $L\in(0,\infty)$ then~\eqref{eq:GinvF} and \eqref{eq:ratiotoL} imply
\begin{align*}
  L &= \lim_{x\rightarrow \infty} \frac{-\log(\E[B]h^*(x))}{x h^*(x)} 
    = \lim_{x\rightarrow \infty} \int_x^{\Ginv(F(x))} \frac{h^*(t)}{xh^*(x)} \dd t \\
    &= \lim_{x\rightarrow \infty} \frac{1}{\log(x)} \int_x^{\Ginv(F(x))} \frac{l(t)-1}{\log t} \cdot \frac{\log(x)}{l(x)-1} \cdot \frac{\log(t)}{t} \dd t \\
    &= \lim_{x\rightarrow \infty} \frac{1}{\log(x)} \int_x^{\Ginv(F(x))} \frac{\log(t)}{t} \dd t 
    = \lim_{x\rightarrow \infty} \frac{\log^2(\Ginv(F(x)))-\log^2(x)}{2\log(x)}.
\end{align*}
Writing $\Ginv(F(x)) = u(x)x, u(x)x \rightarrow \infty,$ now yields
\begin{align*}
  L &= \lim_{x\rightarrow \infty} \log(u(x))\left(1+\frac{\log(u(x))}{2\log(x)}\right),
\end{align*}
from which we conclude $u(x)\rightarrow e^L$ and consequently $\lim_{x\rightarrow\infty} \Ginv(F(x))/x = e^L$. 
%Solving $log^2(u)+2log(x)log(u)-2Llog(x)=0 yields log(u) = -log(x)\pm log(x)\sqrt{1+2L/log(x)}. Since log(xu) \rightarrow \infty, it must be + instead of -. As such, lim_{x\rightarrow \infty} log(u(x)) = lim_{x\rightarrow \infty} (\sqrt{1+2L/log(x)}-1)log(x) = lim_{y\rightarrow \infty} (\sqrt{1+2L/y}-1)/(1/y)) = lim_{y\rightarrow \infty} (1/2*(1+2L/y)^{-1/2}*2L*(-1/y^2))/(-1/y^2) = L.

Finally, if $L=\infty$ then $h^*(x)\downarrow 0$ and therefore $\Ginv(F(x))\geq x$ by \eqref{eq:GinvF}. For sake of contradiction, assume $\liminf_{x\rightarrow\infty} \Ginv(F(x))/x < \infty$. Then there exists $M_0\geq 1$ such that for all $M\geq M_0$ there exists a sequence $(x_n)_{n\in\N}, x_n\rightarrow\infty,$ such that $\Ginv(F(x_n))/x_n \leq M$ for every $n\in\N$. A similar analysis as in (i) then shows that this contradicts relation \eqref{eq:ratiotoL}, and therefore $\lim_{x\rightarrow\infty} \Ginv(F(x))/x = \infty$.
%
%As such,
%\begin{align*}
%  \frac{-\log(\E[B]h^*(x_n))}{x_n h^*(x_n)}
%    &= \int_{x_n}^{\Ginv(F(x_n))} \frac{h^*(t)}{x_n h^*(x_n)} \dd t
%    = \int_1^{\Ginv(F(x_n))/x_n} \frac{h^*(\tau x_n)}{h^*(x_n)} \dd \tau \\
%    &\leq \int_1^M \frac{h^*(\tau x_n)}{h^*(x_n)} \dd \tau
%    \sim \int_1^M \tau^{\a-1} \dd \tau 
%    = \a^{-1} (M^\a-1);
%\end{align*}
%however, taking on both sides the limit as $n\rightarrow\infty$ yields a contradiction.
%
\end{enumerate}

\subsection{Finite support}
Now assume $x_R<\infty$. Theorem~\ref{thm:MDAGumbel} states that $\barF(x)$ can be represented as \[\barF(x) = c(x) \exp\left\{-\int_z^x g(t) h^*(t) \dd t\right\}, \quad z<x<x_R,\] where $c$ and $g$ are measurable functions satisfying $c(x)\rightarrow c>0, g(t)\rightarrow 1$ as $x\uparrow x_R$, and the auxiliary function $f_F(\cdot)=1/h^*(\cdot)$ is positive, absolutely continuous and has density $f_F'(x)$ satisfying $\lim_{x\uparrow x_R} f_F'(x) = 0$. It is easily verified that the function $\barF_\infty(x):= \barF(x_R-x^{-1}), x\geq (x_R-z)^{-1},$ is also in $\MDA(\Lambda)$ with auxiliary function $f_\infty(x):= x^2/h^*(x_R-x^{-1})$. 
%\begin{align*}
%  c_\infty(x) &:= c(x_R-x^{-1}) \rightarrow c \\
%  g_\infty(x) &:= g(x_R-x^{-1}) \rightarrow 1 \\
%  a_\infty(t) &:= x^2/h^*(x_R-x^{-1}) \\
%  z_\infty &:= 1/(x_R-z) \\
%  \barF_\infty(x) &= c_\infty(x) \exp\left\{-\int_{z_\infty}^x g_\infty(t) /a_\infty(t) \dd t\right\}, x>z_\infty, \\
%    &= c(x_R-x^{-1}) \exp\left\{-\int_{1/(x_R-z)}^x g(x_R-t^{-1}) h^*(x_R-t^{-1})/t^2 \dd t\right\}, x>1/(x_R-z), \\
%    &= c(x_R-x^{-1}) \exp\left\{-\int_z^{x_R-x^{-1}} g(\tau) h^*(\tau) \dd \tau\right\}, x_R-x^{-1}>z, \\
%    &= \barF(x) \\
%  \frac{\dd}{\dd x} a_\infty(x) &= \frac{2x}{h^*(x_R-x^{-1})} + \frac{(h^*)'(x_R-x^{-1})}{(h^*(x_R-x^{-1}))^2} \\
%  \lim_{x\rightarrow \infty} \frac{\dd}{\dd x} a_\infty(x) 
%    &= \lim_{y\uparrow x_R} \frac{2}{(x_R-y)h^*(y)} - \lim_{y\uparrow x_R}\frac{\dd}{\dd y} \frac{1}{h^*(y)} 
%    = 0.
%\end{align*}
From this representation it is easy to obtain a finite-support equivalent of Theorem~\ref{thm:beirland}:
%Note that $\a l(x) x^{\a-1} \sim \barF_\infty(x)/\E[B_\infty]\barG_\infty(x) = h^*_\infty(x)\sim 1/a_\infty(x) = h^*(x_R-x^{-1})/x^2$.
\begin{corollary}\label{cor:beirland}
Assume $x_R<\infty$.
\begin{enumerate}[(i)]
\item If there exists $\a>0$ and a slowly varying function $l(x)$ such that $-\log \barF(x_R-x^{-1})\sim l(x) x^\a$ as $x\rightarrow\infty$, then $h^*(x_R-x^{-1})\sim \a l(x) x^{\a+1}$ as $x\rightarrow\infty$ if and only if 
  \begin{equation}
    \lim_{\l \downarrow 1} \liminf_{x\rightarrow\infty} \inf_{t\in[1,\l]} \{\log h^*(x_R-(tx)^{-1}) - \log h^*(x_R-x^{-1}) - 2\log(t)\} \geq 0.
  \end{equation}
\item If there exists a function $l(x):[0,\infty)\rightarrow \R, \liminf_{x\rightarrow \infty} l(x) >1$ such that for all $\l>0$
  \begin{equation}
    \lim_{x\rightarrow\infty} \frac{-\log\barF(x_R-(\l x)^{-1})+\log\barF(x_R-x^{-1})}{l(x)} = \log(\l),
  \end{equation}
  then $l(x)$ is slowly varying and $h^*(x_R-x^{-1})\sim (l(x)-1)x$ as $x\rightarrow\infty$.
\end{enumerate}
\end{corollary}

%Proefschrift: overweeg finite case bewijs te verwijderen, dubbelop

Again, the cases in Theorem~\ref{thm:TmeanGumbel} correspond to the cases in Corollary~\ref{cor:beirland}. The proof for the finite support case is similar to the infinite support case, yet we state it for completeness. Note that $h^*(x) \rightarrow \infty$ as $x\uparrow x_R$ in both cases, and therefore $\frac{x_R-\Ginv(F(x))}{x_R-x} \geq 1$ for all $x$ sufficiently close to $x_R$ by \eqref{eq:GinvF}.
\begin{enumerate}[(i)]
\item
Assume $h^*(x_R-x^{-1}) \sim \a l(x) x^{\a+1},\a>0,$ and note that
\begin{align*}
  \lim_{x\uparrow x_R} \frac{-\log(\E[B]h^*(x))}{(x_R-x) h^*(x)} 
    &= \lim_{y\rightarrow \infty} \frac{-\log(\E[B]h^*(x_R-y^{-1}))}{h^*(x_R-y^{-1})/y} \\
    &= \lim_{y\rightarrow \infty} \frac{-\log(\E[B]\a l(y)) - (\a+1)\log(y)}{\a l(y)y^{\a}} 
    = 0.
\end{align*}
%and
%\begin{align*}
%  \lim_{x\uparrow x_R} \frac{-\log(\E[B]h^*(x))}{(x-x_R) h^*(x)} 
%    &= \lim_{n\rightarrow \infty} \frac{-\log(\E[B]h^*(x_R-n^{-1}))}{(x_R-n^{-1})h^*(x_R-n^{-1})} \\
%	 &= \lim_{y\rightarrow \infty} \left[\frac{-\log(\E[B])-\log(l(y))}{k(x)\log(1/(x-x_R))} + \frac{\log(x-x_R)}{x h^*(x)}\right] \\
%    &= \lim_{n\rightarrow \infty} \frac{\log(\Finv(1-n^{-1}))}{\Finv(1-n^{-1}) h^*(\Finv(1-n^{-1}))}
%    = \lim_{n\rightarrow \infty} \frac{c_n\log(d_n)}{d_n} 
%    = L,
%\end{align*}
We will show that $\lim_{x\rightarrow \infty} \frac{x_R-\Ginv(F(x))}{x_R-x} = 1$ by contradiction. By our previous remark, we only need show $\limsup_{x\rightarrow \infty} \frac{x_R-\Ginv(F(x))}{x_R-x} \leq 1$. If this is false, then there exists $\e\in(0,1)$ and a sequence $(x_n)_{n\in\N}, x_n\uparrow x_R$, such that $\frac{x_R-x_n}{x_R-\Ginv(F(x_n))} \leq 1-\e$ for all $n\in\N$. As before, the Uniform Convergence Theorem \citep[Theorems~1.2.1 and 1.5.2]{bingham1989regular} then implies
\begin{align*}
  -\frac{\log(\E[B]h^*(x_n))}{(x_R-x_n) h^*(x_n)}
    &= \int_{x_n}^{\Ginv(F(x_n))} \frac{h^*(t)}{(x_R-x_n) h^*(x_n)} \dd t \\
    &= \int_{\frac{x_R-x_n}{x_R-\Ginv(F(x_n))}}^1 \frac{h^*(x_R-(x_R-x_n)\tau^{-1})}{\tau^2 h^*(x_R-(x_R-x_n))} \dd \tau \\
    &\geq \int_{1-\e}^1 \frac{h^*(x_R-(x_R-x_n)\tau^{-1})}{\tau^2 h^*(x_R-(x_R-x_n))} \dd \tau 
    \sim \int_{1-\e}^1 \tau^{\a-1} \dd \tau 
    = \a^{-1} (1-(1-\e)^\a)
\end{align*}
for every $n\in\N$, which contradicts with $\lim_{x\uparrow x_R} \frac{\log(\E[B]h^*(x))}{(x_R-x) h^*(x)} = 0$.

\item
Now, assume $h^*(x_R-x^{-1})\sim (l(x)-1)x$ and let $L=\lim_{x\rightarrow\infty} \log(x)/l(x) \in [0,\infty]$. Lemma~\ref{lem:xfx} implies $l(x)\rightarrow \infty$, so that
\begin{align*}
  \lim_{x\uparrow x_R} \frac{-\log(\E[B]h^*(x))}{(x_R-x) h^*(x)} 
%    &= \lim_{y\rightarrow \infty} \frac{-\log(\E[B]h^*(x_R-y^{-1}))}{h^*(x_R-y^{-1})/y} \\
    &= \lim_{y\rightarrow \infty} \frac{-\log(\E[B](l(y)-1))-\log(y)}{l(y)-1} 
    = -L. \numberthis \label{eq:ratiotoL2}
\end{align*}
If $L=0$, then $\lim_{x\rightarrow \infty} \frac{x_R-\Ginv(F(x))}{x_R-x} = 1 = e^0$ by the analysis in (i). Alternatively, if $L\in(0,\infty)$ then~\eqref{eq:GinvF} and \eqref{eq:ratiotoL2} imply
\begin{align*}
  L &= \lim_{x\uparrow x_R} \frac{\log(\E[B]h^*(x))}{(x_R-x) h^*(x)} 
    = \lim_{x\uparrow x_R} \int_{\Ginv(F(x))}^x \frac{h^*(t)}{(x_R-x)h^*(x)} \dd t \\
    &= \lim_{x\uparrow x_R} \int_{\frac{1}{x_R-\Ginv(F(x))}}^{\frac{1}{x_R-x}} \frac{h^*(x_R-\tau^{-1})}{(x_R-x)\tau^2 h^*(x_R-(x_R-x))} \dd \tau \\
    &= \lim_{x\uparrow x_R} \frac{1}{\log((x_R-x)^{-1})} \int_{\frac{1}{x_R-\Ginv(F(x))}}^{\frac{1}{x_R-x}} \frac{l(\tau)-1}{\log(\tau)} \cdot \frac{\log((x_R-x)^{-1})}{l((x_R-x)^{-1})-1} \cdot \frac{\log(\tau)}{\tau}\dd \tau \\
    &= \lim_{x\uparrow x_R} \frac{1}{\log(x_R-x)} \int_{\frac{1}{x_R-x}}^{\frac{1}{x_R-\Ginv(F(x))}} \frac{\log(\tau)}{\tau}\dd \tau \\
    &= \lim_{x\uparrow x_R} \frac{\log^2(x_R-\Ginv(F(x)))-\log^2(x_R-x)}{2\log(x_R-x)}
\end{align*}
Write $\Ginv(F(x)) = x_R-(x_R-x)u(x)$ where $(x_R-x)u(x) \rightarrow 0$ for all $x$ sufficiently close to $x_R$. One then obtains
\begin{align*}
  L &= \lim_{x\uparrow x_R} \log(u(x))\left(1+\frac{\log(u(x))}{2\log(x_R-x)}\right),
\end{align*}
implying $u(x)\rightarrow e^L$ and subsequently $\lim_{x\rightarrow\infty} \frac{x_R-\Ginv(F(x))}{x_R-x} = e^L$. 
%Solving $log^2(u)+2log(x-x_R)log(u)-2Llog(x-x_R)=0 yields log(u) = -log(x-x_R)\pm log(x-x_R)\sqrt{1+2L/log(x-x_R)}. Since log((x-x_R)u) \rightarrow -\infty, it must be - instead of +. As such, lim_{x\rightarrow \infty} log(u(x)) = lim_{x\uparrow x_R} -(\sqrt{1+2L/log(x-x_R)}+1)log(x-x_R) = lim_{y\rightarrow \infty} (\sqrt{1-2L/y}+1)/(1/y) = lim_{y\rightarrow \infty} (1/2*(1-2L/y)^{-1/2}*(-2L)*(1/y^2))/(-1/y^2) = L.

Lastly, consider $L=\infty$ and assume $\limsup_{x\rightarrow\infty} \frac{x_R-\Ginv(F(x))}{x_R-x} < \infty$ for sake of contradiction. Then there exists $M_0\geq 1$ such that for all $M\geq M_0$ there exists a sequence $(x_n)_{n\in\N}, x_n\uparrow x_R,$ such that $\frac{x_R-\Ginv(F(x_n))}{x_R-x_n}\leq M$ for every $n\in\N$. A similar analysis as in (i) then shows that this contradicts relation \eqref{eq:ratiotoL2}, and therefore $\lim_{x\rightarrow\infty} \frac{x_R-\Ginv(F(x))}{x_R-x} = \infty$.
%
%As such,
%\begin{align*}
%  \frac{\log(\E[B]h^*(x_n))}{(x_R-x_n) h^*(x_n)}
%    &= \int_{\Ginv(F(x_n))}^{x_n} \frac{h^*(t)}{(x_R-x_n) h^*(x_n)} \dd t \\
%    &= \int_{\frac{x_R-x_n}{x_R-\Ginv(F(x_n))}}^1 \frac{h^*(x_R-(x_R-x_n)\tau^{-1})}{\tau^2 h^*(x_R-(x_R-x_n))} \dd \tau \\
%    &\leq \int_{1/M}^1 \frac{h^*(x_R-(x_R-x_n)\tau^{-1})}{\tau^2 h^*(x_R-(x_R-x_n))} \dd \tau \\
%    &\sim \int_{1/M}^1 \tau^{\a-1} \dd \tau 
%    = \a^{-1} (1-M^{-\a})
%\end{align*}
%however, taking on both sides the limit as $n\rightarrow\infty$ yields a contradiction.
\end{enumerate}

\subsection{Proof of Lemma~\ref{lem:strictlyincreasingF}}
For any positive, non-increasing $\phi:[0,1)\rightarrow (0,1)$ that vanishes as the argument tends to unity, we may define
\begin{equation}
  F_{\phi}(x) := \left\{\begin{array}{ll}
  F(x) & \text{ if } x<s_1, \text{ and } \\
  F(x) + \frac{x-s_{n}}{s_{n+1}-s_{n}} (F(s_{n+1})-F(x)) \quad & \text{ if } s_{n} \leq x < s_{n+1}, n\geq 1
  \end{array}\right.
\end{equation}
where $s_1:=0$ and $s_{n+1} := \inf\left\{x\geq 0: F(x) \geq \frac{F(s_n)+\phi(F(s_n))}{1+\phi(F(s_n))}\right\}$ forms a strictly increasing sequence. %A similar construction should work if $F$ is not continuous, using $\lim_{x\uparrow x_R} \barF(x)/\barF(x-)=1$ by \citep[Corollary~1.6]{resnick1987extreme}. 
Now, if $s_n\uparrow s^* < x_R$ then $F(s_n)\uparrow p$ for some $p\in(0,1)$ and therefore, for any $\e\in(0,1)$ and all $n$ sufficiently large, we have $(1-\e)p\leq F(s_n)\leq p$. Consequently, $s_{n+1}$ must satisfy $p\geq F(s_{n+1}) \geq \frac{(1-\e)p+\phi(p)}{1+\phi(p)}$, which yields a contradiction if $\e<\phi(p)\frac{1-p}{p}$. We conclude that $F_{\phi}$ is a strictly increasing, continuous c.d.f.\ that satisfies $\barF_\phi(x)\leq \barF(x)$ for all $x$.

Define $n(x):=\sup\{n\in\N: s_{n-1} \leq x\}$. Then
\begin{align*}
  \frac{\barF_{\phi}(x)}{\barF(x)} 
    &= 1 - \frac{x-s_{n(x)}}{s_{n(x)+1}-s_{n(x)}} \frac{F(s_{n(x)+1})-F(x)}{\barF(x)} 
    \geq 1 - \frac{F(s_{n(x)+1})-F(s_{n(x)})}{1-F(s_{n(x)+1})} \\
    &\geq 1 - \frac{\frac{F(s_{n(x)})+\phi(F(s_{n(x)}))}{1+\phi(F(s_{n(x)}))}-F(s_{n(x)})}{1-\frac{F(s_{n(x)})+\phi(F(s_{n(x)}))}{1+\phi(F(s_{n(x)}))}}
    = 1 - \phi(F(s_{n(x)}))
    \rightarrow 1 \numberthis \label{eq:barFphioverbarF}
\end{align*}
as $x\uparrow x_R$, so that $\barF_{\uparrow}(x)\sim\barF(x)$ by our earlier remark.

Let $(s_n)_{n\in\N}$ and $(\tilde{s}_n)_{n\in\N}$ be the sequences associated with $F_{\phi}$ and $F_{\tilde{\phi}}$ and assume $\tilde{\phi}(y)\leq \phi(y)$ for all $y\in[0,1)$. We prove $\tilde{s}_n\leq s_n$ for all $n\in\N$ by induction. The inequality $\tilde{s}_1\leq s_1$ is immediate from the definition. Now, assume that $\tilde{s}_n\leq s_n$ and observe that $(F(s)+q)/(1+q)$ is non-decreasing in $s$ for every $q\geq 0$, and in $q$ for every $s\in\R$. Thus, any $x$ that satisfies $F(x) \geq (F(s_n)+\phi(F(s_n)))/(1+\phi(F(s_n)))$ evidently satisfies $F(x) \geq (F(\tilde{s}_n)+\tilde{\phi}(F(\tilde{s}_n)))/(1+\tilde{\phi}(F(\tilde{s}_n)))$ and hence $\tilde{s}_{n+1}\leq s_{n+1}$. 

As $F_{\phi}(x) \geq F(x)$ implies $\Ginv(F_{\phi}(x)) \geq \Ginv(F(x))$, the proof is complete once we show that there is a version of $\phi$ such that $\limsup_{x\uparrow x_R} \frac{\Ginv(F(x))}{\Ginv(F_{\phi}(x))}\geq 1$. To this end, we  construct a suitable $\phi$ inductively. 

Fix $\phi_1:=1/2$. Then, for $n=1,2,\ldots$, let $r_{n+1}:=\inf\left\{x\geq 0: F(x) \geq \frac{F(s_n)+\phi_n}{1+\phi_n}\right\}$, denote $\phi_{n+1} := \min\{\phi_n, 2^{-2}\E[B]^{-2}\barF(\Ginv(r_{n+1}))^2\}$ and define $\phi(y) := \phi_{n+1}$ for $y\in [F(s_n), F(s_{n+1}))$.

Since $\phi(F(s_n))\leq \phi_n$, it must be that $s_n\leq r_n$ for all $n\in\N$. As a consequence, $\phi(F(s_n)) \leq 2^{-2}\E[B]^{-2}\barF(\Ginv(s_{n+1}))^2$. Writing $\eta(x):=\phi(F(s_{n(x)}))$ for notational convenience, one may now use \eqref{eq:barFphioverbarF} to deduce
\begin{align*}
  \Ginv(F(x)) &= \inf\{z\in\R: G(z) \geq F(x) \}
    = \inf\{z\in\R: \barG(z) \leq \barF(x) \} \\
    &\geq \inf\left\{z\in\R: \barG(z) \leq \frac{\barF_{\phi}(x) }{1-\eta(x)} \right\} \\
    &= \inf\left\{z-\sqrt{\eta(x)}\in\R: \barG(z) + \E[B]^{-1} \int_{z-\sqrt{\eta(x)}}^z \barF(t)\dd t \leq \barF_{\phi}(x) + \frac{\eta(x)}{1-\eta(x)} \barF_{\phi}(x)\right\} \\
    &\geq \inf\left\{z\in\R: \barG(z) + \E[B]^{-1} \sqrt{\eta(x)} \barF(z) \leq \barF_{\phi}(x) + \frac{\eta(x)}{1-\eta(x)} \right\} - \sqrt{\eta(x)} \\
    &\geq \Ginv(F_{\phi}(x)) - \sqrt{\eta(x)},
\end{align*}
where the last inequality follows from the relation
\begin{align*}
  \frac{\eta(x)}{1-\eta(x)} - \E[B]^{-1} \sqrt{\eta(x)} \barF(z) \hspace{-15pt}&\hspace{15pt} %15pt
    \leq \frac{\phi(F(s_{n(x)}))}{1-\phi(F(s_{n(x)}))} - \E[B]^{-1} \sqrt{\phi(F(s_{n(x)}))} \barF(\Ginv(F(s_{n(x)+1}))) \\
    &\leq \sqrt{\phi(F(s_{n(x)}))}\left[2\sqrt{\phi(F(s_{n(x)}))} - \E[B]^{-1} \barF(\Ginv(F(s_{n(x)+1})))\right]
    \leq 0
\end{align*}
for all $z\leq \Ginv(F_{\phi}(x))\leq \Ginv(F(s_{n(x)+1}))$. We conclude that $\barG(F_{\uparrow}(x))\sim \barG(F(x))$ as $x\uparrow x_R$.

\section{Scaled sojourn time tends to zero in probability}
\label{sec:ToverET}
The current section is dedicated to the proof of Theorem~\ref{thm:ToverET}. The intuition behind the proof is that the sojourn time of all jobs of size at most $\tilde{x}_\r$ grows slower than $\E[T_\fb]$, where $\tilde{x}_\r$ is a function that depends on $F$. Alternatively, the fraction of jobs of size at least $\tilde{x}_\r$ tends to zero, since $\tilde{x}_\r\rightarrow x_R$ as $\rhotoone$. Section~\ref{sec:Ttail} discusses the sojourn time of these jobs in more detail.

For any $\e>0$ we have
\begin{equation}
  \P\left(\frac{T_\fb}{\E[T_\fb]} > \e \right) 
    = \int_0^\infty \P(T_\fb(x) > \e \E[T_\fb]) \dd F(x)
    \leq \P(T_\fb(\tilde{x}_\r) > \e \E[T_\fb]) + \barF(\tilde{x}_\r),
  \label{eq:ToverETmain}
\end{equation}
where the final term vanishes as $\rhotoone$ by choice of $\tilde{x}_\r$. The proof is completed if the first probability at the right-hand side also vanishes as $\rhotoone$. 

In preparation of the analysis of $\P(T_\fb(\tilde{x}_\r) > \e \E[T_\fb])$, reconsider the busy period representation $T_\fb(x) \stackrel{d}{=} \mathcal{L}_x(W_x + x)$. The relation states that the sojourn time of a job of size $x$ is equal in distribution to a busy period with job sizes $B_i\wedge x$, initiated by the job of size $x$ itself and the time $W_x$ required to serve all jobs already in the system up to level $x$. Here, the random variable $W_x$ is equal in distribution to the steady state waiting time in an $\mg/\fifo$ queue with job sizes $B\wedge x$.

Let $N_x(t)$ denote a Poisson process with rate $\r_x/\E[B\wedge x]$. Then, it follows from the busy period representation of $T_\fb$ that
\begin{align*}
  \P((1-\r)^2T_\fb(x) > y)
    &= \P(\mathcal{L}_x(W_x + x) > (1-\r)^{-2}y) \\
    &= \P\left(\inf\left\{t\geq 0: \sum_{i=1}^{N(t)}(B_i\wedge x)-t \leq -(W_x+x)\right\}> (1-\r)^{-2}y\right) \\
    &= \P\left(\inf_{t\in[0,(1-\r)^{-2}y]}\left\{\sum_{i=1}^{N(t)}(B_i\wedge x)-t\right\} \geq -(W_x+x)\right) \\
    &= \P\left(\sup_{t\in[0,y]} \left\{\frac{t}{(1-\r)^2}-\sum_{i=1}^{N((1-\r)^{-2}t)}(B_i\wedge x) \right\}\leq W_x+x\right). \numberthis \label{eq:busyperiodrepresentation}
\end{align*}
%Relation~\ref{eq:busyperiodrepresentation} shows that $T_\fb(x_1)\leq_{st} T_\fb(x_2)$ for all $x_1\leq x_2$. 
Additionally, application of Chebychev's inequality to the above relation yields
\begin{align*}
  \P((1-\r)^2T_\fb(x) > y)
    &\leq \P\left(\frac{y}{(1-\r)^2}- \sum_{i=1}^{N((1-\r)^{-2}y)}(B_i\wedge x) \leq W_x+x\right) \\
    &\leq \P\bigg(\bigg|W_x + \sum_{i=1}^{N((1-\r)^{-2}y)}(B_i\wedge x) - \frac{\r_x}{1-\r_x}\E[(B\wedge x)^*] - \frac{\r_x}{(1-\r)^2}y\bigg| \geq \\
      &\hspace{6cm} \frac{1-\r_x}{(1-\r)^2}y-x-\frac{\r_x}{1-\r_x}\E[(B\wedge x)^*] \bigg) \\
    &\leq \frac{\Var[W_x]+\Var\left[\sum_{i=1}^{N((1-\r)^{-2}y)}(B_i\wedge x)\right]}{\left(\frac{1-\r_x}{(1-\r)^2}y-x-\frac{\r_x}{1-\r_x}\E[(B\wedge x)^*]\right)^2} \\
%    &= \frac{\frac{\r_x^2}{(1-\r_x)^2} \E[(B\wedge x)^*]^2 + \frac{\r_x}{1-\r_x}\E[((B\wedge x)^*)^2]+\E[N((1-\r)^{-2}y)]\Var[(B\wedge x)] + \Var[N(y/(1-\r)^2] \E[B\wedge x]^2}{\left(\frac{1-\r_x}{(1-\r)^2}y-x-\frac{\r_x}{1-\r_x}\E[(B\wedge x)^*]\right)^2} \\
%    &= \frac{\frac{\r_x^2}{(1-\r_x)^2} \E[(B\wedge x)^*]^2 + \frac{\r_x}{1-\r_x}\E[((B\wedge x)^*)^2]+ \frac{\l \Var[(B\wedge x)]}{(1-\r)^2} y + \frac{\l \E[B\wedge x]^2}{(1-\r)^2} y}{\left(\frac{1-\r_x}{(1-\r)^2}y-x-\frac{\r_x}{1-\r_x}\E[(B\wedge x)^*]\right)^2} \\
%    &= \frac{\frac{\r_x^2}{(1-\r_x)^2} \E[(B\wedge x)^*]^2 + \frac{\r_x}{1-\r_x}\E[((B\wedge x)^*)^2]+ \frac{\l \E[(B\wedge x)^2]}{(1-\r)^2} y}{\left(\frac{1-\r_x}{(1-\r)^2}y-x-\frac{\r_x}{1-\r_x}\E[(B\wedge x)^*]\right)^2} \\
    &= \frac{\frac{\r_x^2}{(1-\r_x)^2} \E[(B\wedge x)^*]^2 + \frac{\r_x}{1-\r_x}\E[((B\wedge x)^*)^2]+ \frac{2\r_x \E[(B\wedge x)^*]}{(1-\r)^2} y}{\left(\frac{1-\r_x}{(1-\r)^2}y-x-\frac{\r_x}{1-\r_x}\E[(B\wedge x)^*]\right)^2}.
    \numberthis \label{eq:ToverETmain2}
\end{align*}
At this point, similar to the approach in Section~\ref{sec:Tmean}, we differentiate between the finite and infinite variance cases.

\subsection{Finite variance}
\label{subsec:ToverETfinite}
This section considers all functions $F$ that satisfy one of the conditions in the theorem statement and have finite variance. Specifically, this excludes the case $x_R=\infty$, $\beta(\barF)>-2$. Fix
\begin{equation}
  \tilde{p}(F) := \left\{ \begin{array}{ll}
    \frac{\beta(\barF)}{\beta(\barF)+1} & \text{ if } F\notin \MDA(\Lambda) \text{ and } x_R=\infty, \\
    \frac{\beta(\barF(x_R-(\cdot)^{-1})}{\beta(\barF(x_R-(\cdot)^{-1}))-1} \quad & \text{ if } F\notin \MDA(\Lambda) \text{ and } x_R<\infty, \text{ and } \\
    1 & \text{ if } F\in\MDA(\Lambda),
  \end{array} \right.
\end{equation}
and $\tilde{\gamma}\in(\tilde{p}(F)/2, 1)$, and define $\nu(\r):=(1-\r)^{\tilde{\gamma}}$ and $\tilde{x}_\r:= x_\r^{\nu(\r)} = \Ginv\left(1-\frac{1-\r}{\r}\frac{1-\nu(\r)}{\nu(\r)}\right)$. Indeed $\tilde{x}_\r\rightarrow x_R$, and we proceed with the analysis in \eqref{eq:ToverETmain2}. Noting that $\E[((B\wedge x)^*)^2] = \frac{\E[(B\wedge x)^3]}{3\E[B]} \leq \frac{x\E[B^2]}{3\E[B]} = \frac{2}{3}\E[B^*]x$ and substituting $x=\tilde{x}_\r$, gives
\begin{align*}
  \P((1-\r)^2T_\fb(\tilde{x}_\r) > y)
    &\leq \frac{\left(\frac{1-\r}{1-\r_{\tilde{x}_\r}}\right)^2\E[B^*]^2 + \frac{1-\r}{1-\r_{\tilde{x}_\r}}\frac{2}{3}\E[B^*](1-\r)\tilde{x}_\r + 2 \E[B^*] y}{\left(\frac{1-\r_{\tilde{x}_\r}}{1-\r}y-(1-\r)\tilde{x}_\r-\frac{1-\r}{1-\r_{\tilde{x}_\r}}\r_{\tilde{x}_\r}\E[B^*]\right)^2} \\
    &= \frac{\E[B^*]^2\nu(\r)^2 + \frac{2}{3}\E[B^*]\nu(\r) (1-\r)x_\r^{\nu(\r)} + 2 \E[B^*] y}{\left(\nu(\r)^{-1}y-(1-\r)x_\r^{\nu(\r)}-\r_{x_\r^{\nu(\r)}}\E[B^*]\nu\right)^2}.
\end{align*}

We now return to the probability $\P(T_\fb(\tilde{x}_\r) > \e \E[T_\fb])$ in relation~\eqref{eq:ToverETmain}. By Theorems~\ref{thm:ETMatuszewska} and \ref{thm:ETextreme}, there exists $C>0$ such that the inequality $(1-\r)^2\E[T_\fb] \geq C \barF(\Ginv(\r))$ holds true for all $\rho$ sufficiently close to one. Denoting $\tilde{\e}:=\e C$, this gives
\begin{align*}
  \P(T_\fb(\tilde{x}_\r) > \e \E[T_\fb])
    &\leq \P((1-\r)^2T_\fb(\tilde{x}_\r) > \tilde{\e} \barF(\Ginv(\r))) \\
    &\leq \frac{\E[B^*]^2\nu(\r)^2 + \frac{2}{3}\E[B^*]\nu(\r) (1-\r)x_\r^{\nu(\r)} + 2 \tilde{\e} \E[B^*] \barF(\Ginv(\r))}{\left(\tilde{\e}\nu(\r)^{-1}\barF(\Ginv(\r))-(1-\r)x_\r^{\nu(\r)}-\r_{x_\r^{\nu(\r)}}\E[B^*]\nu(\r)\right)^2} \\
    &= \frac{\E[B^*]^2\frac{\nu(\r)^4}{\barF(\Ginv(\r))^2} + \frac{2\E[B^*]}{3} \frac{\nu(\r)^3 (1-\r)x_\r^{\nu(\r)}}{\barF(\Ginv(\r))^2} + 2 \tilde{\e} \E[B^*]\frac{\nu(\r)^2}{\barF(\Ginv(\r))}}{\left(\tilde{\e}-\frac{\nu(\r)(1-\r)x_\r^{\nu(\r)}}{\barF(\Ginv(\r))}-\r_{x_\r^{\nu(\r)}}\E[B^*]\frac{\nu(\r)^2}{\barF(\Ginv(\r))}\right)^2}.
\end{align*}
Subsequently, we observe for any $\nu\in(0,1)$ that
\begin{equation}
  \lim_{\rhotoone} (1-\r) x_\r^{\nu} 
    = \lim_{\rhotoone} (1-\r) \Ginv\left(1-\frac{1-\nu}{\nu}\frac{1-\r}{\r}\right) 
    = \lim_{z\rightarrow x_R} \frac{\frac{\nu}{1-\nu}\barG(z) \cdot z}{1+\frac{\nu}{1-\nu}\barG(z)}
    \leq \lim_{z\rightarrow x_R} \frac{\nu\cdot z\barG(z)}{1-\nu}.
  \label{eq:zbarGz}
\end{equation}
where $z\barG(z)\rightarrow 0$ as $z \rightarrow x_R$ since $\E[B^2]<\infty$ (cf.\ Section~\ref{subsubsec:finitevariance}). It follows that $(1-\r) x_\r^{\nu(\r)} = o(\nu(\r))$ as $\rhotoone$, so that $\lim_{\rhotoone} \P(T_\fb > \e \E[T_\fb]) = 0$ provided that $\lim_{\rhotoone} \frac{\nu(\r)^2}{\barF(\Ginv(\r))}=0$.

Write $x=(1-\r)^{-1}$. By Lemma~\ref{lem:onesidedMatuszewskavanish}, it suffices to show $\a\left((\cdot)^{-2\tilde{\gamma}} \barF(\Ginv(1-(\cdot)^{-1}))\right) < 0$. This relation follows from Lemma~\ref{lem:onesidedMatuszewskaproduct}, Corollary~\ref{cor:FGinvMatuszewska} and our choice of $\tilde{\gamma}$:
\begin{align*}
  \a\left(\frac{(\cdot)^{-2\tilde{\gamma}} }{ \barF(\Ginv(1-(\cdot)^{-1})) }\right)
    &\leq -2\tilde{\gamma} - \beta\left(\barF(\Ginv(1-(\cdot)^{-1}))\right)
    \leq -2\tilde{\gamma} + \tilde{p}(F)
    < 0.
\end{align*}
%As we have seen in Sections~\ref{subsubsec:ETFrechet} and \ref{subsubsec:ETGumbel}, $\barF(\Ginv(1-(\cdot)^{-1}))$ is regularly varying with index $-p(H)$. Therefore, we fix $\tilde{\d}\in(0,2\tilde{\gamma} - p(H))$ and apply Potter's Theorem \citep[Theorem~1.5.6]{bingham1989regular} to obtain the inequality $\barF(\Ginv(\r))\geq C(1-\r)^{p(H)+\d}$ for some constant $C>0$ and all $\rho$ sufficiently close to one. Consequently, $\lim_{\rhotoone} \frac{\nu_l(\r)^2}{\barF(\Ginv(\r))} \leq \lim_{\rhotoone} C(1-\r)^{2\tilde{\gamma}-p(H)-\d} = 0$.

\subsection{Infinite variance}
This section regards all functions $F$ that satisfy $x_R=\infty, \beta(\barF)>-2$. In this case, $\tilde{x}_\r$ can be any function that satisfies both $\lim_{\rhotoone} \tilde{x}_\r = \infty$ and $\lim_{\rhotoone} \frac{\tilde{x}_\r}{\barG(\tilde{x}_\r) \log\left(\frac{1}{1-\r}\right)} = 0$. 

Theorem~\ref{thm:ETMatuszewska} implies that there exists $C>0$ such that $\E[T_\fb]\geq C \log\left(\frac{1}{1-\r}\right)$ for all $\r$ sufficiently close to one. Again, denote $\tilde{\e}=\e C$. The analysis resumes with relation~\eqref{eq:ToverETmain2}, where we substitute $y$ by $\tilde{\e} (1-\r)^2 \log\left(\frac{1}{1-\r}\right)$ to obtain
\begin{align*}
  \P(T_\fb(x) > \e \E[T_\fb])
    &\leq \P\left((1-\r)^2T_\fb(x) > \tilde{\e} (1-\r)^2 \log\left(\frac{1}{1-\r}\right)\right) \\
    &\leq \frac{\frac{1}{(1-\r_x)^2} \E[(B\wedge x)^*]^2 + \frac{1}{1-\r_x}\E[((B\wedge x)^*)^2]+ 2 \tilde{\e} \E[(B\wedge x)^*] \log\left(\frac{1}{1-\r}\right)}{\left(\tilde{\e}(1-\r_x)\log\left(\frac{1}{1-\r}\right)-x-\frac{\r_x}{1-\r_x}\E[(B\wedge x)^*]\right)^2}.
\end{align*}
By relation~\eqref{eq:m2x}, there exists a function $b(x)$ that is bounded for all $x$ sufficiently large and satisfies $m_2(x) = \E[B] b(x) x \barG(x)$. As such, $\E[((B\wedge x)^*)^2]=\frac{\E[(B\wedge x)^3]}{3\E[B]}\leq \frac{x m_2(x)}{3\E[B]}=b(x)x^2 \barG(x)/3$ and similarly $\E[(B\wedge x)^*] = \frac{m_2(x)}{2\E[B]} = b(x)x \barG(x)/2$. Substituting this into the above relation yields
\begin{align*}
  \P(T_\fb(x) > \e \E[T_\fb])
    &\leq \frac{\frac{b(x)^2}{4}\frac{x^2 \barG(x)^2}{(1-\r_x)^2} + \frac{b(x)}{3}\frac{x^2 \barG(x)}{1-\r_x}+ \tilde{\e} b(x) x \barG(x) \log\left(\frac{1}{1-\r}\right)}{\left(\tilde{\e}(1-\r_x)\log\left(\frac{1}{1-\r}\right)-x-\frac{\r_x b(x)}{2}\frac{x \barG(x)}{1-\r_x}\right)^2},
\end{align*}
so that
\begin{multline*}
  \P(T_\fb(x) > \e \E[T_\fb]) \\
    \leq \frac{\frac{b(x)^2}{4}\frac{\barG(x)^2}{(1-\r_x)^2}\frac{x^2}{(1-\r_x)^2\log^2\left(\frac{1}{1-\r}\right)} + \frac{b(x)}{3}\frac{\barG(x)}{1-\r_x}\frac{x^2}{(1-\r_x)^2\log^2\left(\frac{1}{1-\r}\right)} + \tilde{\e} b(x) \frac{\barG(x)}{1-\r_x}\frac{x}{(1-\r_x)\log\left(\frac{1}{1-\r}\right)}}{\left(\tilde{\e}-\frac{x}{(1-\r_x)\log\left(\frac{1}{1-\r}\right)}-\frac{\r_x b(x)}{2}\frac{\barG(x)}{1-\r_x}\frac{x}{(1-\r_x)\log\left(\frac{1}{1-\r}\right)} \right)^2}.
\end{multline*}
The result follows after noting that $1-\r_x = 1-\r G(x) \geq \barG(x)$ and substituting $\tilde{x}_\r$ for $x$.

\section{Asymptotic behaviour of the sojourn time tail}
\label{sec:Ttail}
In this section, we prove Theorem~\ref{thm:Ttail} after presenting two facilitating propositions. The proofs of the propositions are postponed to Sections~\ref{subsec:proofWtoExp} and \ref{subsec:proofTtail}. Throughout this section, $\mathbf{e}(q)$ will denote an Exponentially distributed random variable with rate $q>0$. We abuse notation by writing $\mathbf{e}(0)=+\infty$.

Reconsider the relation $T_\fb(x) \stackrel{d}{=} \mathcal{L}_x(W_x + x)$ to gain some intuition. A rough approximation of the duration of a busy period, given $W_x + x$ units of work at time $t=0$, is $(W_x + x)/(1-\r_x)$. The scaled sojourn time $(1-\r)^2T_\fb(x)$ is then approximated by $\frac{1-\r}{1-\r_x}(1-\r)(W_x + x)$. As in Section~\ref{sec:Tmean}, define $x_\r^\nu = \Ginv\left(1-\frac{1-\r}{\r}\frac{1-\nu}{\nu}\right),\nu\in(1-\r,1),$ so that $\frac{1-\r}{1-\r_x}=\nu$. Then for all $\nu\in(0,1)$, we have $(1-\r)^2T_\fb(x_\r^\nu) \stackrel{d}{\approx} \nu(1-\r)(W_{x_\r^\nu} + x_\r^\nu)$. We will show that $(1-\r)x_\r^\nu\rightarrow 0$ for all fixed $\nu\in(0,1)$. Instead, the following proposition shows that $(1-\r)W_{x_\r^\nu}$ behaves as an exponentially distributed random variable as $\rhotoone$:

\begin{proposition} \label{prop:WtoExp}
Let $x_\r^\nu = \Ginv\left(1-\frac{1-\r}{\r}\frac{1-\nu}{\nu}\right),\nu\in(1-\r,1),$ and let $W_x^\r$ denote the steady state waiting time in an $\mg/\fifo$ queue with job sizes $B_i\wedge x$ and arrival rate $\r_x/\E[B\wedge x]$. Then, for any fixed $\nu\in(0,1)$,
$%\begin{equation}
  (1-\r)W_{x_\r^\nu} \stackrel{d}{\rightarrow} \Exp((\nu\E[B^*])^{-1})
%  \label{eq:WtoExp}
$%\end{equation} 
as $\rhotoone$.
\end{proposition}

If $W^\r=W_\infty^\r$ denotes the steady state waiting time in the non-truncated system, then \citet{kingman1961single} proved that $(1-\r)W^\r\stackrel{d}{\rightarrow} \Exp(\E[B^*]^{-1})$. Proposition~\ref{prop:WtoExp} shows how jobs can be truncated such that the exponential behaviour is preserved, and quantifies how the truncation affects the parameter of the exponential distribution. 

Substituting the result in Proposition~\ref{prop:WtoExp} into our approximation yields $(1-\r)^2T_\fb(x_\r^\nu) \stackrel{d}{\approx} \Exp((\nu^2 \E[B^*])^{-1})$ for every fixed $\nu\in(0,1)$. We will show that the fraction of jobs for which $\nu$ is in $(\e,1-\e)$ scales as $\barF(\Ginv(\r))$, and that the contribution of other jobs to the tail of $(1-\r)^2T_\fb$ is negligible. The result is presented in Proposition~\ref{prop:Ttail}, where we focus on the probability $\P((1-\r)^2 T_\fb > \mathbf{e}(q))$ for its connection to the Laplace transform of $T_\fb^*$.
%and that the sojourn time that these jobs experience scales as $(1-\r)^{-2}$. Additionally, we will show that the contribution to the scaled sojourn time $(1-\r)^2T_\fb$ of smaller jobs is negligible, as is the contribution of jobs for which $(1-\r)x_\r^\nu$ does not vanish.

\begin{proposition} \label{prop:Ttail}
Assume $F\in\MDA(H)$, where $H$ is an extreme value distribution. Let $p(H)=\frac{\a}{\a-1}$ if $H=\Phi_\a,\a>2$; $p(H)=1$ if $H=\Lambda$ and $p(H)=\frac{\a}{\a+1}$ if $H=\Psi_\a,\a>0$. Then
\begin{equation}
  \lim_{\rhotoone} \frac{\P((1-\r)^2 T_\fb > \mathbf{e}(q))}{\barF(\Ginv(\r))} 
    = \int_0^1 \frac{8\E[B^*]q \nu}{\sqrt{1+4\E[B^*]q \nu^2}\left(\sqrt{1+4\E[B^*]q \nu^2}+1\right)^2} \left(\frac{1-\nu}{\nu}\right)^{p(H)} \dd \nu
\end{equation}
for all $q\geq 0$. Here, the integral is finite for all $q\geq 0$.
\end{proposition}

We are now ready to prove Theorem~\ref{thm:Ttail}. Using the relation $\E[e^{-qY}] = \P(\mathbf{e}(q)>Y)$, 
%\begin{align*}
%  \E[e^{-qY}] &= \int_0^\infty e^{-qy} \dd \P(Y\leq y)
%    = \int_0^\infty \P(\mathbf{e}(q) \geq y) \dd \P(Y\leq y)
%    = \P(\mathbf{e}(q)\geq Y),
%\end{align*}
one sees that $\P((1-\r)^2 T^\r_{\fb} > \mathbf{e}(q)) = 1-\E[e^{-q(1-\r)^2 T^\r_{\fb}}]$ and consequently
\begin{align*}
  \frac{\P((1-\r)^2T_\fb> \mathbf{e}(q))}{\barF(\Ginv(\r))}
    &= \frac{(1-\r)^2 \E[T_\fb]}{\barF(\Ginv(\r))} \cdot \frac{1-\E\left[e^{-q(1-\r)^2 T_\fb}\right]}{(1-\r)^2 \E[T_\fb]} \\
    &= \frac{(1-\r)^2 \E[T_\fb]}{\barF(\Ginv(\r))} \cdot q \cdot \E\left[e^{-q(1-\r)^2 T_\fb^*}\right],
\end{align*}
where $T_\fb^*$ is the residual sojourn time and has density $\P(T_\fb>t)/\E[T_\fb]$. 
%
%\begin{align*}
%  q \E\left[e^{-q(1-\r)^2 T_\fb^*}\right]
%    &= q \int_0^\infty e^{-q(1-\r)^2 t} \dd \P(T_\fb^*\leq t) \\
%    &= q \E[T_\fb]^{-1} \int_0^\infty e^{-q(1-\r)^2 t} \P(T_\fb > t) \dd t \\
%    &= -(1-\r)^{-2}\E[T_\fb]^{-1} \int_0^\infty  \P(T_\fb > t) \dd e^{-q(1-\r)^2 t} \\
%    &= -(1-\r)^{-2}\E[T_\fb]^{-1} \left[\P(T_\fb > t) e^{-q(1-\r)^2 t}\right]_{t=0}^\infty \\
%      &\qquad - (1-\r)^{-2}\E[T_\fb]^{-1} \int_0^\infty  e^{-q(1-\r)^2 t} \dd \P(T_\fb \leq t) \\
%    &= (1-\r)^{-2}\E[T_\fb]^{-1} \left(1-\E\left[e^{-q(1-\r)^2 T_\fb}\right]\right).
%\end{align*}
%
%Daarom zal $\E[e^{-(1-\r)^2 T_\fb^*}]$ in die gevallen ook convergeren. Het enige dat we dan nodig hebben is dat het limietpunt differentieerbaar is in $q=0$ om te concluderen dat $(1-\r)^2 T_\fb^*$ zelf ook convergeert. 
Consequently,
\begin{align*}
  \lim_{\rhotoone} \E&\left[e^{-q(1-\r)^2 T_\fb^*}\right] \\
    &= \lim_{\rhotoone} \frac{\barF(\Ginv(\r))}{(1-\r)^2 \E[T_\fb]}  \int_0^1 \frac{8\E[B^*] \nu}{\sqrt{1+4\E[B^*]q \nu^2}\left(\sqrt{1+4\E[B^*]q \nu^2}+1\right)^2} \left(\frac{1-\nu}{\nu}\right)^{p(H)} \dd \nu \\
    &= r(H)^{-1} \int_0^1 \frac{8\nu}{\sqrt{1+4\E[B^*]q \nu^2}\left(\sqrt{1+4\E[B^*]q \nu^2}+1\right)^2} \left(\frac{1-\nu}{\nu}\right)^{p(H)} \dd \nu \numberthis \label{eq:LaplaceTstar}
\end{align*}
for all $q\geq 0$, where $r(H)$ was introduced in Theorem~\ref{thm:ETextreme}. It follows from Sections~\ref{subsubsec:ETFrechet} and \ref{subsubsec:ETGumbel} that $\lim_{q\downarrow 0}\lim_{\rhotoone} \E\left[e^{-q(1-\r)^2 T_\fb^*}\right]=1$. Additionally, the right-hand side is continuous in $q$, so that $(1-\r)^2 T_\fb^*$ converges to some non-degenerate random variable by the Continuity Theorem. %Feller, Thm 2a in Chap XIII.

The Laplace transform inversion formula (12) in \citet[p.234]{bateman1954tables} states that $f(t) = \frac{2}{\sqrt{\pi}}\sqrt{t} - 2te^t \erfc(\sqrt{t})$ is the Laplace inverse of $s^{-1/2}(s^{1/2}+1)^{-2}$,\ i.e., 
$%\begin{equation}
  \int_0^\infty e^{-qt} f(t) \dd t = \frac{1}{\sqrt{q}\left(\sqrt{q}+1\right)^2}.
$ %\end{equation}
Consequently, we have
\begin{equation}
  \int_0^\infty e^{-qt} g(t,\nu) \dd t
    = \frac{1}{\sqrt{1+4\E[B^*]q \nu^2}\left(\sqrt{1+4\E[B^*]q \nu^2}+1\right)^2}
\end{equation}
for 
$%\begin{align*}
  g(t,\nu) 
    = \frac{e^{-\frac{t}{4\E[B^*]\nu^2}}}{4\E[B^*]\nu^2} f\left(\frac{t}{4\E[B^*]\nu^2}\right), %\\
%    &= \frac{e^{-\frac{t}{4\E[B^*]\nu^2}}}{4\E[B^*]\nu^2}  \left(
%    \frac{\sqrt{t}}{\nu \sqrt{\pi \E[B^*]}} - \frac{t}{2\E[B^*]\nu^2}e^{\frac{t}{4\E[B^*]\nu^2}} \erfc\left(\frac{1}{2\nu}\sqrt{\frac{t}{\E[B^*]}}\right)\right)
$ %\end{align*}
and hence relation~\eqref{eq:LaplaceTstar} may be rewritten as
\begin{align*}
  \lim_{\rhotoone} \E\left[e^{-q(1-\r)^2 T_\fb^*}\right]
    %&= \lim_{\rhotoone} \int_0^\infty e^{-q(1-\r)^2t} \frac{\P(T_\fb>t)}{\E[T_\fb]} \dd t \\
    %&= \lim_{\rhotoone} \int_0^\infty e^{-q\tau} \frac{\P((1-\r)^2T_\fb>\tau)}{(1-\r)^2\E[T_\fb]} \dd \tau \\
    &= \int_0^\infty e^{-qt} \left[\int_0^1 8 r(H)^{-1} \nu g(t,\nu) \left(\frac{1-\nu}{\nu}\right)^{p(H)} \dd \nu \right] \dd t 
    =: \int_0^\infty e^{-qt} g^*(t) \dd t.
\end{align*}
We conclude that the limiting random variable $\lim_{\rhotoone} (1-\r)^2T_\fb^*$ has density $g^*$. Furthermore, as 
\begin{align*}
  \lim_{\rhotoone} \E\left[e^{-q(1-\r)^2 T_\fb^*}\right]
    &= \lim_{\rhotoone} \int_0^\infty e^{-q\tau} \frac{\P((1-\r)^2T_\fb>\tau)}{(1-\r)^2\E[T_\fb]} \dd \tau \\
    &= \lim_{\rhotoone} \int_0^\infty e^{-q\tau} \frac{\P((1-\r)^2T_\fb>\tau)}{r(H)\E[B^*]\barF(\Ginv(\r))} \dd \tau,
\end{align*}
for all $q\geq 0$, we also see that $\lim_{\rhotoone} \frac{\P((1-\r)^2T_\fb>y)}{r(H)\E[B^*]\barF(\Ginv(\r))} = g^*(y)$ almost everywhere. 

To see that $g^*$ is monotone, it suffices to show that $f(t)$ is monotone. To this end, we exploit the continued fraction representation (13.2.20a) in \citet{cuyt2008handbook} and find
\begin{align}
  \erfc(x) &= \frac{x}{\sqrt{\pi}} e^{-x^2} \cfrac{1}{x^2+\cfrac{1/2}{1+\cfrac{1}{x^2+\cfrac{3/2}{1+\dots}}}} 
    \geq %\frac{x}{\sqrt{\pi}} e^{-x^2} \cfrac{1}{x^2+\cfrac{1/2}{1+\cfrac{1}{x^2+3/2}}} \\
%    &= \frac{x}{\sqrt{\pi}} e^{-x^2} \frac{2x^2+5}{2x^4+6x^2+3/2} 
    %= 
    \frac{e^{-x^2}}{x\sqrt{\pi}}  \left(1-\frac{x^2+3/2}{2x^4+6x^2+3/2}\right).
\end{align}
As a consequence, one sees that
\begin{align*}
  \frac{\dd}{\dd t} f(t)
    &= \frac{1+2t}{\sqrt{\pi}\sqrt{t}} - 2(1+t)e^t \erfc(\sqrt{t}) 
    \leq \frac{1+2t - 2(1+t)\left(1-\frac{t+3/2}{2t^2+6t+3/2}\right)}{\sqrt{\pi}\sqrt{t}} 
    = \frac{-1+\frac{2t^2+5t+3}{2t^2+6t+3/2}}{\sqrt{\pi}\sqrt{t}},
    %< 0
\end{align*}
which is negative for all $t\geq 0$. We conclude the section with the postponed proofs of Propositions~\ref{prop:WtoExp} and \ref{prop:Ttail}.

\subsection{Proof of Proposition~\ref{prop:WtoExp}}
\label{subsec:proofWtoExp}
The Pollaczek-Khintchine formula states that $\E[e^{-s(1-\r) W_x}] = \frac{1-\r_x}{1-\r_x\E[e^{-s(1-\r)(B\wedge x)^*}]}$. In this representation, we expand the Laplace-Stieltjes transform $\E[e^{-s(1-\r)(B\wedge x)^*}]$ around $\r=1$ to find
\begin{align*}
  \E[e^{-s(1-\r)W_x}] = \frac{1-\r_x}{1-\r_x\left(1-\E[(B\wedge x)^*] (1-\r)s + o(1-\r) \right)}
\end{align*}
and hence
\begin{align*}
  \E[e^{-s (1-\r)W_{x_\r^\nu}}] 
    &= \frac{1}{1+ \frac{1-\r}{1-\r_{x_\r^\nu}} \r_{x_\r^\nu}\E[(B\wedge x_\r^\nu)^*] s + o\left(\frac{1-\r}{1-\r_{x_\r^\nu}}\right)} 
    = \frac{1}{1+ \nu \r_{x_\r^\nu}\E[(B\wedge x_\r^\nu)^*] s + o(1)},
\end{align*}
where $o(1)$ vanishes as $\rhotoone$. By definition of $x_\r^\nu$, $x_\r^\nu\rightarrow \infty$ and $\r_{x_\r^\nu}\uparrow 1$ as $\rhotoone$ for any fixed $\nu\in(0,1)$. In particular, $\lim_{\rhotoone} \E[e^{-s (1-\r)W_{x_\r^\nu}}] = \frac{1}{1+\nu \E[B^*] s}$. The proof is completed by applying the Continuity Theorem.

\subsection{Proof of Proposition~\ref{prop:Ttail}}
\label{subsec:proofTtail}
We require functions $\nu_l(\r)\downarrow 0$ and $\nu_u(\r)\uparrow 1$ that distinguish the jobs that significantly contribute to the tail of $(1-\r)^2T_\fb$, and those that don't. For the former function, fix $\g\in(p(H)/2,1)$ and let $\nu_l(\r)= (1-\r)^{\g}$ as in Section~\ref{subsec:ToverETfinite}. This is possible as $p(H)<2$ for all $H$ to which the theorem applies. For the latter function, we refer to relation~\eqref{eq:zbarGz} to verify that there exists a function $\nu(\r)\uparrow 1$ such that $(1-\r)x_\r^{\nu(\r)} \rightarrow 0$. Let $\nu_u(\r)$ be a function with this property, and write
\begin{align*}
  \frac{\P((1-\r)^2 T_\fb > \mathbf{e}(q))}{\barF(\Ginv(\r))} \hspace{-104pt} & \hspace{104pt} \\
    &= \int_{\nu=0}^{\nu_l(\r)} \P((1-\r)^2 T_\fb(x_\r^\nu) > \mathbf{e}(q)) \frac{\dd F(x_\r^\nu)}{\barF(\Ginv(\r))} 
      + \int_{\nu=\nu_l(\r)}^{\nu_u(\r)} \P((1-\r)^2 T_\fb(x_\r^\nu) > \mathbf{e}(q)) \frac{\dd F(x_\r^\nu)}{\barF(\Ginv(\r))} \\
    &\qquad + \int_{\nu=\nu_u(\r)}^1 \P((1-\r)^2 T_\fb(x_\r^\nu) > \mathbf{e}(q)) \frac{\dd F(x_\r^\nu)}{\barF(\Ginv(\r))} 
    =: \hat{\mathrm{I}}(\r) + \hat{\mathrm{II}}(\r) + \hat{\mathrm{III}}(\r). \numberthis
\end{align*}
The next paragraphs study the behaviour of $\P((1-\r)^2 T_\fb(x) > \mathbf{e}(q))$, which will then facilitate the analysis of the above three regions. Specifically, we will derive the asymptotic behaviour of $\hat{\mathrm{II}}(\r)$ in terms of $q$, and show that $\hat{\mathrm{I}}(\r)+\hat{\mathrm{III}}(\r)=o(1)$ for any $q\geq 0$. 

From relation~\eqref{eq:busyperiodrepresentation}, we know that 
\begin{equation}
  \P((1-\r)^2T_\fb(x) > \mathbf{e}(q))
    = \P\left(\sup_{t\in[0,\mathbf{e}(q)]} X_x^\r(t) \leq (1-\r)W_x+(1-\r)x\right),
\end{equation}
where $X_x^\r(t):=\frac{t}{1-\r}-\sum_{i=1}^{N((1-\r)^{-2}t)}(1-\r)(B_i\wedge x)$. Then $X_x^\r(t)$ is a spectrally negative \levy process, and therefore relation~(8.4) in \citet{kyprianou2014introductory} implies
\begin{equation}
  \P((1-\r)^2 T^\r_{\fb}(x) > \mathbf{e}(q)) 
    = \P\left(\mathbf{e}(\Phi(x,\r,q)) \leq (1-\r)W_x + (1-\r)x \right), \label{eq:tailexponentialtime}
\end{equation}
where $\Phi(x,\r,q) := \sup\{s\geq 0: \psi(x,\r,s) = q\}$ is the right inverse of the Laplace exponent $\psi(s) := t^{-1} \log \E[e^{s X^\r_x(t)}]$ of $X_x^\r(t)$. Since
\begin{align*}
  \psi(x,\r,s) 
    &= t^{-1} \log \E\left[e^{\frac{st}{1-\r} - \sum_{i=1}^{N((1-\r)^{-2}t)} (1-\r)s(B_i\wedge x)}\right] \\
    &= \frac{s}{1-\r} + t^{-1} \log \E\left[e^{- \sum_{i=1}^{N((1-\r)^{-2}t)} (1-\r)s(B_i\wedge x)}\right]
\end{align*}
and
\begin{align*}
  \E[e^{- \sum_{i=1}^{N((1-\r)^{-2}t)} (1-\r)s(B_i\wedge x)}]
    &= \sum_{n=0}^\infty \E[e^{-(1-\r)s(B\wedge x)}]^n \frac{\left(\frac{\l t}{(1-\r)^2}\right)^n}{n!} e^{-\frac{\l t}{(1-\r)^2}} \\
    &= e^{-\frac{\l t}{(1-\r)^2}\left(1-\E[e^{-(1-\r)s(B\wedge x)}]\right)},
\end{align*}
we obtain $\psi(x,\r,s) = \frac{s}{1-\r} - \frac{\l}{(1-\r)^2}\left(1-\E[e^{-(1-\r)s(B\wedge x)}]\right)$. A Taylor expansion around $\r=1$ now yields 
\begin{align*}
  \psi(x,\r,s) \hspace{-26pt} & \\
    &= \frac{s}{1-\r} - \frac{\l}{(1-\r)^2}\left(1-\left(1-(1-\r)s\E[B\wedge x] + \frac{(1-\r)^2s^2}{2}\E[(B\wedge x)^2] + o((1-\r)^2s^2) \right)\right) \\
    &= \frac{s}{1-\r} - \left(\frac{\r_x s}{1-\r} - \frac{\l\E[(B\wedge x)^2]}{2}s^2 + o(s^2)\right) 
    = \frac{1-\r_x}{1-\r}s + \frac{\r\E[(B\wedge x)^2]}{2\E[B]}s^2 + o(s^2),
\end{align*}
so that $\lim_{\rhotoone} \psi(x_\r^\nu,\r,s) = \nu^{-1} s + \E[B^*]s^2$ for all $\nu>0$, and consequently
\begin{equation}
  \lim_{\rhotoone} \Phi(x_\r^\nu,\r,q) = \frac{\sqrt{\nu^{-2}+4\E[B^*]q}-\nu^{-1}}{2\E[B^*]}
    = \frac{\sqrt{1+4\E[B^*]q \nu^2}-1}{2\E[B^*]\nu}
    =: \Phi(\nu,q).
  \label{eq:phinonzero}
\end{equation}
Similarly, one deduces that $\lim_{\rhotoone} \nu_l(\r) \psi(x_\r^{\nu_l(\r)},\r,s) = s$ and
\begin{equation}
  \lim_{\rhotoone} \nu_l(\r)^{-1}\Phi(x_\r^{\nu_l(\r)},\r,q) = q.
  \label{eq:phizero}
\end{equation}
We now gathered sufficient tools to analyse the asymptotic behaviour of $\hat{\mathrm{II}}(\r)$. 

Fix $\e\in(0,1/3)$. We have already shown that $(1-\r)W^\r_{x_\r^\nu} \rightarrow \mathbf{e}((\nu \E[B^*])^{-1})$ and $(1-\r)x_\r^\nu \rightarrow 0$ as $\rhotoone$ for all $\nu \in(0,1)$. Since $\mathbf{e}(q_1)\leq_{st} \mathbf{e}(q_2)$ whenever $q_1\geq q_2$, relations~\eqref{eq:tailexponentialtime} and \eqref{eq:phinonzero} imply
\begin{align*}
  \P((1-\r)^2 T^\r_{\fb}(x_\r^\nu) > \mathbf{e}(q)) \hspace{-60pt} & \hspace{60pt}
    \leq \P\left(\mathbf{e}((1+\e)\Phi(\nu,q)) \leq \mathbf{e}((1-\e)(\nu \E[B^*])^{-1}) + \e \right) \\
    &= e^{-(1+\e)\e\Phi(\nu,q)}\frac{(1+\e)\Phi(\nu,q)}{(1+\e)\Phi(\nu,q) + (1-\e)(\nu \E[B^*])^{-1}} + 1 - e^{-\e(1+\e)\Phi(\nu,q)} \\
    &\leq \frac{\sqrt{1+4\E[B^*]q \nu^2}-1}{\sqrt{1+4\E[B^*]q \nu^2}+1-\frac{4\e}{1+\e}} + 1 - e^{-\e\cdot \frac{\sqrt{1+4\E[B^*]q\nu^2}-1}{\E[B^*]\nu}}
\end{align*}
for all $\rho\geq \rho_\e$, $\rho_\e$ sufficiently close to one. Consequently, for all $\rho\geq \rho_\e$,
\begin{align*}
  \hat{\mathrm{II}}(\r)
    &\leq \int_{\nu_l(\r)}^{\nu_u(\r)} \frac{\sqrt{1+4\E[B^*]q \nu^2}-1}{\sqrt{1+4\E[B^*]q \nu^2}+1-\frac{4\e}{1+\e}} \frac{\dd F\left(\Ginv\left(1-\frac{1-\r}{\r}\frac{1-\nu}{\nu}\right)\right)}{\barF(\Ginv(\r))} \\
      &\qquad + \int_{\nu_l(\r)}^{\nu_u(\r)} \left(1 - e^{-\e\cdot \frac{\sqrt{1+4\E[B^*]q\nu^2}-1}{\E[B^*]\nu}}\right) \frac{\dd F\left(\Ginv\left(1-\frac{1-\r}{\r}\frac{1-\nu}{\nu}\right)\right)}{\barF(\Ginv(\r))} \\
    &\leq -\left[\frac{\sqrt{1+4\E[B^*]q \nu^2}-1}{\sqrt{1+4\E[B^*]q \nu^2}+1-\frac{4\e}{1+\e}} \frac{\barF\left(\Ginv\left(1-\frac{1-\r}{\r}\frac{1-\nu}{\nu}\right)\right)}{\barF(\Ginv(\r))} \right]_{\nu=\nu_l(\r)}^{\nu_u(\r)} \\
      &\qquad + \int_{\nu_l(\r)}^{\nu_u(\r)} \frac{8\E[B^*]q \nu}{\sqrt{1+4\E[B^*]q \nu^2}\left(\sqrt{1+4\E[B^*]q \nu^2}+1-\frac{4\e}{1+\e}\right)^2} \frac{\barF\left(\Ginv\left(1-\frac{1-\r}{\r}\frac{1-\nu}{\nu}\right)\right)}{\barF(\Ginv(\r))} \dd \nu \\
        &\qquad \qquad - \left[\left(1 - e^{-\e\cdot \frac{\sqrt{1+4\E[B^*]q\nu^2}-1}{\E[B^*]\nu}}\right) \frac{\barF\left(\Ginv\left(1-\frac{1-\r}{\r}\frac{1-\nu}{\nu}\right)\right)}{\barF(\Ginv(\r))}\right]_{\nu=\nu_l(\r)}^{\nu_u(\r)}  \\
          &\qquad \qquad \qquad + 4q \int_{\nu_l(\r)}^{\nu_u(\r)} \e \cdot e^{-\e\cdot\frac{\sqrt{1+4\E[B^*]q\nu^2}-1}{\E[B^*]\nu}} \frac{\barF\left(\Ginv\left(1-\frac{1-\r}{\r}\frac{1-\nu}{\nu}\right)\right)}{\barF(\Ginv(\r))} \dd \nu.
\end{align*}
In Sections~\ref{subsubsec:ETFrechet} and \ref{subsubsec:ETGumbel}, we deduced that $\barF(\Ginv(1-(\cdot)^{-1}))$ is regularly varying with index $-p(H)$. The Uniform Convergence Theorem then implies
\begin{align*}
  \limsup_{\rhotoone} \hat{\mathrm{II}}(\r)
    &\leq -\left[\frac{\sqrt{1+4\E[B^*]q \nu^2}-1}{\sqrt{1+4\E[B^*]q \nu^2}+1-\frac{4\e}{1+\e}} \left(\frac{1-\nu}{\nu}\right)^{p(H)} \right]_{\nu=0}^1 \\
      &\qquad + \int_0^1 \frac{8\E[B^*]q \nu}{\sqrt{1+4\E[B^*]q \nu^2}\left(\sqrt{1+4\E[B^*]q \nu^2}+1-\frac{4\e}{1+\e}\right)^2} \left(\frac{1-\nu}{\nu}\right)^{p(H)} \dd \nu \\
        &\qquad \qquad - \left[\left(1 - e^{-\e\cdot \frac{\sqrt{1+4\E[B^*]q\nu^2}-1}{\E[B^*]\nu}}\right) \left(\frac{1-\nu}{\nu}\right)^{p(H)}\right]_{\nu=0}^1  \\
          &\qquad \qquad \qquad + 4q \int_0^1 \e \cdot e^{-\e\cdot \frac{\sqrt{1+4\E[B^*]q\nu^2}-1}{\E[B^*]\nu}} \left(\frac{1-\nu}{\nu}\right)^{p(H)} \dd \nu \\ 
    &= \int_0^1 \frac{8\E[B^*]q \nu}{\sqrt{1+4\E[B^*]q \nu^2}\left(\sqrt{1+4\E[B^*]q \nu^2}+1-\frac{4\e}{1+\e}\right)^2} \left(\frac{1-\nu}{\nu}\right)^{p(H)} \dd \nu \\
      &\qquad + 4q \int_0^1 \e \cdot e^{-\e\cdot \frac{\sqrt{1+4\E[B^*]q\nu^2}-1}{\E[B^*]\nu}} \left(\frac{1-\nu}{\nu}\right)^{p(H)} \dd \nu.
\end{align*}
%since $\lim_{\nu\downarrow 0} \frac{\sqrt{1+4\E[B^*]q \nu^2}-1}{\nu^{\a/(\a-1)}} = \lim_{\nu\downarrow 0} \frac{4\E[B^*]q \nu}{\sqrt{1+4\E[B^*]q \nu^2}} \frac{\a-1}{\a}\nu^{-1/(\a-1)} = 0$ by l'H{\^o}pital. 
%
Now, both integrals are bounded for all $\e\in(0,1/3)$ and all $q\geq 0$. Additionally, both integrands are increasing in $\e$ for all $\e$ sufficiently small. One may thus take the limit $\e\downarrow 0$ and apply the Dominated Convergence Theorem to find
\begin{equation}
  \limsup_{\rhotoone} \hat{\mathrm{II}}(\r) 
    \leq \int_0^1 \frac{8\E[B^*]q \nu}{\sqrt{1+4\E[B^*]q \nu^2}\left(\sqrt{1+4\E[B^*]q \nu^2}+1\right)^2} \left(\frac{1-\nu}{\nu}\right)^{p(H)} \dd \nu.
\end{equation}

Similarly, one may show that
\begin{equation*}
  \liminf_{\rhotoone} \hat{\mathrm{II}}(\r)
    \geq \int_0^1 \frac{8\E[B^*]q \nu}{\sqrt{1+4\E[B^*]q \nu^2}\left(\sqrt{1+4\E[B^*]q \nu^2}+1+\frac{4\e}{1-\e}\right)^2} \left(\frac{1-\nu}{\nu}\right)^{p(H)} \dd \nu,
\end{equation*}
and we conclude
\begin{equation}
  \lim_{\rhotoone} \hat{\mathrm{II}}(\r)
  = \int_0^1 \frac{8\E[B^*]q \nu}{\sqrt{1+4\E[B^*]q \nu^2}\left(\sqrt{1+4\E[B^*]q \nu^2}+1\right)^2} \left(\frac{1-\nu}{\nu}\right)^{p(H)} \dd \nu.
\end{equation}

Second, consider $\hat{\mathrm{I}}(\r)$. Define $M(\rho):=(1-\r)^{-\hat{\gamma}}$ for some $\hat{\gamma}\in\left(p(H)/2,\gamma\right)$ and recall that $(1-\r)x\rightarrow 0$ and $(1-\r)W_x^\r \stackrel{d}{\rightarrow} 0$ (and hence in probability) for all $x\leq x_\r^{\nu_l(\r)}$. Thus, for all $x\leq x_\r^{\nu_l}$ and all $\rho$ sufficiently large, we have
\begin{align*}
  \hat{\mathrm{I}}(\r)
    &= \int_{\nu=0}^{\nu_l(\r)} \P\left(\mathbf{e}(\Phi(x_\r^\nu,\r,q)) \leq (1-\r)W_{x_\r^\nu} + (1-\r)x_\r^\nu \right) \frac{\dd F(x_\r^\nu)}{\barF(\Ginv(\r))} \\
    &\leq \frac{\P\left(\mathbf{e}(\Phi(x_\r^{\nu_l(\r)},\r,q)) \leq 2 M(\r) \right)}{\barF(\Ginv(\r))} + \frac{\P((1-\r)W^\r_x \geq M(\r))}{\barF(\Ginv(\r))}
    =: \hat{\mathrm{Ia}}(\r) + \hat{\mathrm{Ib}}(\r)
\end{align*}
Fix $\d\in\left(0,p(H)-\gamma-\hat{\gamma} \right)$. Potter's Theorem \citep[Theorem~1.5.6]{bingham1989regular} states that $\barF(\Ginv(\r))\geq C(1-\r)^{p(H)+\d}$ for some constant $C>0$ and all $\rho$ sufficiently close to one. Also, one may readily deduce from relation~\eqref{eq:phizero} that $\mathbf{e}(\Phi(x_\r^{\nu_l(\r)},\r,q))\geq_{st} \mathbf{e}(2q\nu_l(\r))$ for all $x\leq x_\r^{\nu_l(\r)}$ and $\rho$ sufficiently large. Consequently,
\begin{align*}
  \limsup_{\rhotoone} \hat{\mathrm{Ia}}(\r)
    &\leq \limsup_{\rhotoone} \frac{1-e^{-4q\nu_l(\r)M(\r)}}{\barF(\Ginv(\r))} 
    \leq \lim_{\rhotoone} \frac{1-e^{-4q(1-\r)^{\gamma-\hat{\gamma}}}}{C(1-\r)^{p(H)+\d}} \\
    &= \lim_{\rhotoone} \frac{4q(\gamma-\hat{\gamma})(1-\r)^{\gamma-\hat{\gamma}-1}e^{-4q(1-\r)^{\gamma-\hat{\gamma}}}}{C\left(p(H)+\d\right)(1-\r)^{p(H)-1+\d}} \\
    &= \lim_{\rhotoone} \frac{4q(\gamma-\hat{\gamma})}{C\left(p(H)+\d\right)} \cdot \exp\left[-4q(1-\r)^{\gamma-\hat{\gamma}} + \left(\gamma-\hat{\gamma}-p(H)-\d\right)\log(1-\r)\right] 
    = 0.
\end{align*}
For term $\hat{\mathrm{Ib}}(\r)$, we apply Markov's inequality and Potter's Theorem to obtain
\begin{align*}
  \limsup_{\rhotoone} \hat{\mathrm{Ib}}(\r)
    &\leq \limsup_{\rhotoone} \frac{\frac{1-\r}{1-\r_x}\r_x\E[(B\wedge x)^*]}{M(\r) \barF(\Ginv(\r))} 
    \leq \lim_{\rhotoone} C_1 \frac{\E[B^*] \nu_l(\r) }{M(\r) (1-\r)^{p(H)+\d}} \\
    &= \lim_{\rhotoone} C_1 \E[B^*] (1-\r)^{\gamma+\hat{\gamma}-p(H)-\d}
    = 0.
\end{align*}

Finally, consider term $\hat{\mathrm{III}}(\r)$. For this term, the claim follows rapidly from the Uniform Convergence Theorem and the property $\nu_u(\r)\uparrow 1$:
\begin{align*}
  \limsup_{\rhotoone} \hat{\mathrm{III}}(\r) 
    &\leq \limsup_{\rhotoone} \frac{\barF(x_\r^{\nu_u})}{\barF(\Ginv(\r))} 
    = \limsup_{\rhotoone} \frac{\barF\left(\Ginv\left(1-\frac{1-\r}{\r}\frac{1-\nu_u(\r)}{\nu_u(\r)}\right)\right)}{\barF(\Ginv(\r))} \\
    &= \limsup_{\rhotoone} \left(\frac{1-\nu_u(\r)}{\r \nu_u(\r)}\right)^{p(H)}
    = 0.
\end{align*}
This concludes the proof.

\section*{Acknowledgements}
The work of Bart Kamphorst is part of the free competition research programme with project number 613.001.219, which is financed by the Netherlands Organisation for Scientific Research (NWO). The research of Bert Zwart is partly supported by the NWO VICI grant 639.033.413.

\appendix
\section{Additional Matuszewska theory}
\label{app:preliminaries}
This appendix gathers some results on Matuszewska indices. Lemmas~\ref{lem:onesidedMatuszewskaproduct} and \ref{lem:onesidedMatuszewskavanish} are proven directly from the definition of the lower and upper Matuszewska indices. Then, a generalized version of Potter's Theorem allows us to prove Lemmas~\ref{lem:barGMatuszewska} and \ref{lem:inverseMatuszewska2}.

\begin{proof}[Proof of Lemma~\ref{lem:onesidedMatuszewskaproduct}]
Let $\a_1>\a(f_1)$ and $\a_2>\a(f_2)$. Then, by definition of the upper Matuszewska index, there exist $C_1, C_2>0$ such that for all $\mu\in[1,\mu^*], \mu^*>1,$ we have $f_1(\mu x) \leq C_1 \mu^{\a_1} f_1(x)$ and $f_2(\mu x) \leq C_2 \mu^{\a_2} f_2(x)$ for all $x$ sufficiently large. Consequently, we have $\limsup_{x\rightarrow\infty} \frac{f_1(\mu x)f_2(\mu x)}{f_1(x)f_2(x)} \leq C_1 C_2 \mu^{\a_1+\a_2}$ and thus $\a(f_1\cdot f_2)\leq \a(f_1)+\a(f_2)$. 

Similarly, if $f_1$ is non-decreasing, we have $f_1(f_2(\mu x)) \leq f_1(C_2 \mu^{\a_2} f_2(x)) \leq C_1 C_2^{\a_1} \mu^{\a_1 \a_2} f_1(f_2(x))$.
and thus $\a(f_1\circ f_2) \leq \a(f_2) \cdot \a(f_2)$. The results on the lower Matuszewska indices are proven analogously.
%
%Let $\beta_1<\beta(f_1)$ and $\beta_2<\beta(f_2)$. Then, by definition of the Matuszewska index, there exists $D_1, D_2>0$ such that for all $\mu\in[1,\mu^*]$ we have $f_1(\mu x) \geq D_1 \mu^{\beta_1} f_1(x)$ and $f_2(\mu x) \geq D_2 \mu^{\beta_2} f_2(x)$ for all $x$ sufficiently large. Consequently, we have $\liminf_{x\rightarrow\infty} \frac{f_1(\mu x)f_2(\mu x)}{f_1(x)f_2(x)} \geq D_1 D_2 \mu^{\beta_1+\beta_2}$ and thus $\beta(f_1\cdot f_2)\geq \beta(f_1)+\beta(f_2)$. 
%
%Similarly, if $f_1$ is non-decreasing, we have $f_1(f_2(\mu x)) \geq f_1(D_2 \mu^{\beta_2} f_2(x)) \geq C_1 C_2^{\beta_1} \mu^{\beta_1 \beta_2} f_1(f_2(x))$ and thus $\a(f_1\circ f_2) \geq \beta(f_2) \cdot \beta(f_2)$.
\end{proof}

\begin{proof}[Proof of Lemma~\ref{lem:onesidedMatuszewskavanish}]
As $f$ is positive, it suffices to show that $\limsup_{x\rightarrow \infty} f(x) = 0$. For sake of contradiction, assume that this is false. Then there exists a constant $m>0$ and a sequence $(x_n)_{n\in\N}, x_n\rightarrow \infty$, such that $f(x_n)\geq m$ for all $n\in\N$. Now, by definition of the upper Matuszewska index, there exists $C>0$ such that for all $\mu\in[1,\mu^*], \mu^*>1,$ we have $f(x) \geq C \mu^{-\a(f)/2} f(\mu x)$ for all $x$ sufficiently large. As a consequence, for some $N\in\N$ we have $f(x_N) \geq C (x_n/x_N)^{-\a(f)/2} f(x_n) \geq C m(x_n/x_N)^{-\a(f)/2}$ for any fixed $n\geq N$. This is a contradiction for any $x_n$ that satisfies $x_n>x_N(Cm/f(x_N))^{2/\a(f)}$.
\end{proof}

The following result is a generalized version of Potter's theorem and gives bounds on the ratio $f(y)/f(x)$:
\begin{theorem}[\citet{bingham1989regular}, Proposition~2.2.1] \label{thm:onesidedPotter}
Let $f$ be positive. 
\begin{enumerate}[(1)]
\item If $\a(f)<\infty$, then for every $\a>\a(f)$ there exist positive constants $C$ and $X$ such that $f(y)/f(x) \leq C (y/x)^\a$ for all $y\geq x\geq X$.
\item If $\b(f)>-\infty$, then for every $\b<\b(f)$ there exist positive constants $D$ and $X$ such that $f(y)/f(x) \geq D (y/x)^\b$ for all $y\geq x\geq X'$.
\end{enumerate}
\end{theorem}

Theorem~\ref{thm:onesidedPotter} allows us to derive a relation between the Matuszewska indices of $f$ to those of $\inv{f}$, which is presented as Lemma~\ref{lem:inverseMatuszewska}:
\begin{lemma}\label{lem:inverseMatuszewska}
Let $f$ be positive and locally integrable on $[X,\infty)$. If $f$ is strictly increasing, unbounded above and $\a(f)<\infty$, then $\beta(\inv{f}) = 1/\a(f)$. %\alert{so far, the lemma is stated as Ex. 8 in Section 2 of \citet{bingham1989regular}, but the main reference is an unpublished manuscript that I cannot find online.} 
If $\beta(f)>0$, then $\a(\inv{f}) = 1/\beta(f)$.
\end{lemma}
\begin{proof}
By definition of the upper Matuszewska index, for all $\a>\a(f)$ there exists a constant $C>0$ such that for each $\mu^*>1$, $f(\mu x)/f(x) \leq C \mu^\a$ uniformly in $\mu\in[1,\mu^*]$ as $\xtoinfty$. In particular, for all $x$ sufficiently large we have $f((\mu/C)^{1/\a} x) \leq \mu f(x)$. As $f$ is strictly increasing and unbounded above, one can hence see that
\begin{align*}
  \lim_{\xtoinfty} \frac{\inv{f}(\mu x)}{\inv{f}(x)} 
    &= \lim_{y \rightarrow \infty} \frac{\inv{f}(\mu f(y))}{\inv{f}(f(y))} 
    \geq \lim_{y \rightarrow \infty} \frac{\inv{f}(f((\mu/C)^{1/\a} y))}{y}%{D^{-1/\beta} y} 
    \geq (C)^{-1/\a} \mu^{1/\a}  \numberthis \label{eq:upperboundfinvf}
\end{align*}
uniformly for $\mu\in[1,\mu^*]$. As a consequence, $\beta(\inv{f})\geq 1/\a(f)$. 

On the other hand, if $\beta(\inv{f})>1/\a(f), \a(f)>0,$ then Theorem~\ref{thm:onesidedPotter}(2) claims that for some $\e>0$ sufficiently small there exists a constant $C'>0$ such that $\inv{f}(y)/\inv{f}(z) \geq C' (y/z)^{1/\a(f)+\e}$ for all $y\geq z$ sufficiently large. Consequently, one may substitute $y=f(\mu x), z=f(x)$ to obtain
\begin{align*}
  C' \left(\frac{f(\mu x)}{f(x)}\right)^{1/\a(f)+\e}
    &\leq \frac{\inv{f}(f(\mu x))}{\inv{f}(f(x))} 
    = \mu
\end{align*}
and hence $\lim_{\xtoinfty} f(\mu x)/f(x) \leq ((C')^{-1}\mu)^{\frac{\a(f)}{1+\e\a(f))}}$. This inequality, however, indicates that $\a(f)$ was not the infimum over all $\a$ satisfying \eqref{eq:upperMatuszewska}, which is a contradiction. 

The relation $\a(\inv{f}) = 1/\beta(f)$ is proven similarly. 
\end{proof}
A more general version of this lemma has been stated in several other works \citep[e.g.\ ][]{bingham1989regular, lin2011heavy}; however, these works refer to an unpublished manuscript by De Haan and Resnick for the corresponding proof. %For completeness, we prove this version and all following lemmas in Appendix~\ref{app:preliminaries}.

Our final results relate the Matuszewska indices of $\barF$ to those of related functions. First, Lemma~\ref{lem:barGMatuszewska} relates the Matuszewska indices of $\barF$ to those of $\barG$. Its proof is similar to the proof of Lemma~6 in \citet{lin2011heavy}.
\begin{proof}[Proof of Lemma~\ref{lem:barGMatuszewska}]
First assume $x_R=\infty$. Then by definition of $\a(\barF)$, we have for all $\a>\a(\barF)$ that $\barF(\mu t)/\barF(t) \leq C (1+o(1)) \mu^\a$ uniformly in $\mu\in[1,\mu^*]$ and hence
\begin{align*}
  \E[B] \barG(\mu x) 
%    &= \int_{\mu x}^\infty \barF(t) \dd t \\
    &= \mu \int_x^\infty \barF(\mu \tau) \dd \tau 
    \leq C (1+o(1)) \mu^{\a+1} \int_x^\infty \barF(\tau) \dd \tau \\
    &= C (1+o(1)) \mu^{\a+1} \E[B] \barG(x)
\end{align*}
as $\xtoinfty$. On the other hand, if $x_R<\infty$ then
\begin{align*}
  \E[B]\barG(x_R-(\mu x)^{-1}) 
    &= \int_{x_R-(\mu x)^{-1}}^{x_R} \barF(t) \dd t 
    = \int_x^\infty \mu^{-1} \tau^{-2} \barF(x_R-(\mu\tau)^{-1}) \dd \tau \\
    &\leq C (1+o(1)) \mu^{\a-1} \int_x^\infty \tau^{-2} \barF(x_R-\tau^{-1}) \dd \tau \\
    &= C (1+o(1)) \mu^{\a-1} \E[B] \barG(x_R-x^{-1})
\end{align*}
as $\xtoinfty$. The claims on the lower Matuszewska index can be proven analogously.
\end{proof}

Second, Lemma~\ref{lem:inverseMatuszewska2} relates the Matuszewska indices of $\barF$ to those of $\Ginv$. It does so by combining Lemmas~\ref{lem:onesidedMatuszewskaproduct}, \ref{lem:inverseMatuszewska} and \ref{lem:barGMatuszewska}.
\begin{proof}[Proof of Lemma~\ref{lem:inverseMatuszewska2}]
We only prove the relation between the lower Matuszewska indices, as the relation between the upper Matuszewska indices can be proven similarly. 

First, assume $x_R=\infty$. Since $\beta(\barF)>-\infty$, it follows from Lemma~\ref{lem:barGMatuszewska} that $\beta(\barG)>-\infty$ and hence, by Lemma~\ref{lem:onesidedMatuszewskaproduct}, that $\a(1/\barG) = -\a(\barG)\leq -\beta(\barG)<\infty$. The result follows readily from Lemma~\ref{lem:inverseMatuszewska} through $\beta(\Ginv(1-(\cdot)^{-1})) = \beta(\inv{(1/\barG)}) = 1/\alpha(1/\barG) = -1/\beta(\barG)$ and subsequent application of Lemma~\ref{lem:barGMatuszewska}. 

Similarly, if $x_R<\infty$ then $\a(1/\barG(x_R-(\cdot)^{-1}))<\infty$ and
\begin{align*}
  \frac{1}{x_R-\Ginv(1-x^{-1})} 
    &= \frac{1}{x_R-\inf\{z: G(z) > 1-x^{-1}\}}
    = \inf\left\{\frac{1}{x_R-z}: G(z) > 1-x^{-1}\right\} \\
    &= \inf\{y: G(x_R-y^{-1}) > 1-x^{-1}\} 
    = \inf\{y: 1/\barG(x_R-y^{-1}) > x\} \\
    &= \inv{\left(\frac{1}{\barG\left(x_R-\frac{1}{\cdot}\right)}\right)}(x).
\end{align*}
The result then follows from $\beta\left(\frac{1}{x_R-\Ginv(1-(\cdot)^{-1})}\right) = \beta\left(\inv{\left(\frac{1}{\barG\left(x_R-(\cdot)^{-1}\right)}(\cdot)\right)}\right) = 1/\alpha\left(\frac{1}{\barG(x_R-(\cdot)^{-1})}\right) = -1/\beta(\barG(x_R-(\cdot)^{-1}))$ and application of Lemma~\ref{lem:barGMatuszewska}.
\end{proof}

\bibliographystyle{apalike}
\DeclareRobustCommand{\NLprefix}[3]{#3}
%\bibliographystyle{amsplain}
%\bibliography{../Bibliography}

\end{document}